\RequirePackage{fix-cm}
\RequirePackage{amsmath}

\documentclass[smallextended,natbib]{svjour3+}       

\makeatletter
\let\cl@chapter\undefined
\makeatletter

\smartqed

\usepackage{bm}
\usepackage{enumitem}
\usepackage{graphicx}
\usepackage{hyperref}
\usepackage{newtxtext,newtxmath}
\usepackage{tikz}
\usepackage{amsmath,amssymb}

\DeclareMathOperator*{\Extreme}{ext}
\DeclareMathOperator{\ArensEells}{\text{{\normalfont\AE}}}
\DeclareMathOperator{\CUT}{\text{{\normalfont CUT}}}
\DeclareMathOperator{\Expectation}{\mathbb E}
\DeclareMathOperator{\KB}{\text{{\normalfont KB}}}
\DeclareMathOperator{\Lipschitz}{Lip}

\DeclareMathOperator{\Varsig}{Sgn}

\newcommand{\AEspace}[1]{\ArensEells\left(#1\right)}

\newcommand{\aval}[1]{\left\vert#1\right\vert}
\newcommand{\chof}[1]{\operatorname{child}\left(#1\right)}

\newcommand{\expectat}[2]{\Expectation_{#1}\left[#2\right]}
\newcommand{\expectof}[1]{\Expectation\left(#1\right)}

\newcommand{\normat}[2]{\left\Vert#2\right\Vert_{#1}}
\newcommand{\normof}[1]{\left\Vert#1\right\Vert}

\newcommand{\reals}{\mathbb{R}}

\newcommand{\scalarof}[2]{\left\langle#1,#2\right\rangle}
\newcommand{\setof}[2]{\left\{#1 \,\middle|\, #2 \right\}}
\newcommand{\set}[1]{\left\{#1\right\}}

\usepackage{lineno}
\newcommand*\patchAmsMathEnvironmentForLineno[1]{%
   \expandafter\let\csname old#1\expandafter\endcsname\csname #1\endcsname
   \expandafter\let\csname oldend#1\expandafter\endcsname\csname end#1\endcsname
   \renewenvironment{#1}%
      {\linenomath\csname old#1\endcsname}%
      {\csname oldend#1\endcsname\endlinenomath}}%
\newcommand*\patchBothAmsMathEnvironmentsForLineno[1]{%
   \patchAmsMathEnvironmentForLineno{#1}%
   \patchAmsMathEnvironmentForLineno{#1*}}%
\AtBeginDocument{%
\patchBothAmsMathEnvironmentsForLineno{equation}%
\patchBothAmsMathEnvironmentsForLineno{align}%
\patchBothAmsMathEnvironmentsForLineno{flalign}%
\patchBothAmsMathEnvironmentsForLineno{alignat}%
\patchBothAmsMathEnvironmentsForLineno{gather}%
\patchBothAmsMathEnvironmentsForLineno{multline}%
}

\usepackage{cleveref}

\begin{document}

\title{Kantorovich distance on a finite metric space}

\thanks{G.~Pistone acknowledges the support of de Castro Statistics and Collegio Carlo Alberto. He is a member of GNAMPA-INdAM.}

\author{Luigi~Montrucchio \and Giovanni~Pistone}

\authorrunning{L.~Montrucchio, G.~Pistone} 

\institute{Luigi Montrucchio \at
Collegio Carlo Alberto,
Piazza Vincenzo Arbarello 8,
10122 Torino, Italy \\
\email{luigi.montrucchio@unito.it}
\and
Giovanni Pistone \at
de Castro Statistics, Collegio Carlo Alberto, 
Piazza Vincenzo Arbarello 8, 
10122 Torino, Italy \\
\email{giovanni.pistone@carloalberto.org}
}

\date{}

\maketitle

\begin{abstract}
  Kantorovich distance (or 1-Wasserstein distance) on the probability simplex of a finite metric space is the value of a Linear Programming problem for which a closed-form expression is known in some cases. When the ground distance is defined by a graph, a few examples have already been studied. 
  In the present paper, after re-deriving, with different tools, the result for trees, we prove that, for an arbitrary weighted graph, the K-distance is the minimum of the K-distances over all the spanning trees associated with the graph. We work in the dual LP-problem by using Arens-Eells norm associated with the metric space. 
  Finally, we introduce new norms that are naturally related to $\ell_1$-embeddable distances and allows for a partial extension of our results to this new setting.
  

\end{abstract}

\section{Introduction}

The set-up for Information Geometry (IG) as defined by S.-I. Amari provides a dually flat vector bundle on top of a statistical model. We refer to \citet{amari|nagaoka:2000}, \citet{amari:2016}, \citet*{ay|jost|le|schwachhofer:2017}. From the non-parametric point of view, one element of the duality is the affine structure whose basis is the full open probability simplex on the sample space $X$ and the (tangent) vector space is the vector space $M_0(X)$ of zero-mass measures. This tangent space is sometimes called the mixture space or m-space.

The choice of a convenient norm on $M_0(X)$ is clearly seminal for the full development of a non-parametric IG. In a sense, the most natural choice is the variation norm. With such a choice, the basic separating duality on which full theory is developed relies on $(\xi,u) \mapsto \int u \ d\xi$ with $\xi \in M_0(X)$ and $u \in L^1_0(X,\xi)$. Other options have been considered, where the approach is still non-parametric but only special classes of measures is taken into account. For example, \citet{pistone|sempi:95}, \citet{pistone:2013GSI}, model the m-tangent space as an Orlicz space of $L\log L$ type.

Both approaches have been criticised from the point of view of applications to Statistics and Machine Learning because the proposed norms induce a topology on the probability simplex which is too strong. It is argued that one would like to get an IG that allows for the use of weak convergence of probability measures.

One possible way to get a norm on $M_0(X)$, such that the corresponding distance on the probability simplex is essentially equivalent to the weak convergence of probability measures, is based on using the distance defined on the probability simplex, whose theory has been fully developed in \citet{kantorovich|rubinstein:1958}. Precisely, we refer here to what is usually called 1-Wasserstein distance in the Optimal Transport (OT) literature. It is a distance which is compatible with the affine structure of the probability simplex and moreover provides a convergence essentially equivalent with the weak convergence. OT is currently a very active field of research, both from the mathematical point of view and as a tool in many applications. A number of monographs have been recently published, among which, we like notably to mention \citet{villani:2003-topics}, \citet{villani:2008optimal}, \citet*{ambrosio|gigli|savare:2008}, \citet{santambrogio:2015OTAP}. The case of finite state spaces and the related computational methods is discussed in detail by \citet{peyre|cuturi:2019}.

With that motivation in mind, in this paper we study the simplest case, where the sample space is a finite metric space with distance $d$. In this case, the convergence in variation and the weak convergence are topologically equivalent, but nevertheless there is an important conceptual difference. We do not consider here the other $p$-Wasserstein distance, with $p > 1$, that have been recently discussed in the IG literature in relation with the proposal of Riemannian metrics other than Fisher metric. For example, see in this journal, \citet*{malago|montrucchio|pistone:2018}, \citet{li|montufar:2018}.

Long before the full development by Kantorovich, the issue of evaluating the dissimilarity between two probability distribution $\mu, \nu \in \Delta(X)$ has been discussed in terms of a hidden random variable $Z$ along with the two functions $x,y$ for which $x(Z) \sim \mu$ and $y(Z) \sim \nu$. The dissimilarity is then defined as the minimum expected value $\expectof{d(x(Z),y(Z))}$. An early example is from  \citet{gini:1914dissomiglianza} that provides a solution in the case of distribution with a finite support in the real numbers.

Clearly, there is always a version of the hidden variable on the product space, so that we have the classical definitions that we reproduce below for reader's convenience.

Let $X$ be a set with $n$ points, with generic points $x, y, z, \dots$. In general, we do not number finite sets. $\Delta(X)$ is the probability simplex on $X$, that is, the set of all probability functions $\mu, \nu, \dots$. The sample space $X$ is endowed with a distance $d$. Frequently, the distance is provided by a weighted graph.

Given a couple $(\mu,\nu)$ of probability functions, recall that a joint probability function $\gamma \in \Delta(X \times X)$ is a \emph{coupling}, if $\mu$ and $\nu$ are the two margins of $\gamma$, respectively. The set of all couplings $\mathcal P(\mu,\nu)$ is a subset of $\Delta(X \times X)$ defined by the $2(n-1)$ independent affine constraints 
\begin{equation*}
 \sum_{y\in X} \gamma(x,y) = \mu(x) \ , \quad  \sum_{x\in X} \gamma(x,y) = \nu(y) \ , \qquad x,y \neq x_0 \ .  
\end{equation*}
In particular, it is a polytope, and, being such, it is the convex combination of its vertices.

\begin{definition}\label{def:K-distance}
Given $\mu,\nu \in \Delta(X)$, the \emph{Kantorovich distance} (K-distance) is defined on $\Delta(X)\times\Delta(X)$ to be the value of the Linear Programming problem
\begin{equation}\label{eq:K-distance}
  d(\mu,\nu) = \inf \setof{\sum_{x,y \in X} d(x,y) \gamma(x,y)}{\gamma \in \mathcal P(\mu,\nu)} \ .
\end{equation}
\end{definition}

This statistical index is a distance that extends the ground distance, i.e., $d(\delta_x,\delta_y) = d(x,y)$. Other properties are easily proved in our finite setting and they are still true in the general (non-finite) setting.

Observe first that the mixture model $\mu(t) = (1-t) \mu + t \nu$, $t \in [0,1]$, is a metric geodesic linking the two probability functions $\mu$ and $\nu$, that is, $d(\mu(s),\mu(t)) = \aval{t-s}d(\mu,\nu)$. Second, the metric convergence is equivalent to the weak convergence. This ensures the compatibility of this set-up both with the affine IG and the required convergence.

It is interesting from the IG point of view to remark that the metric geodesics for the K-distance are not uniquely defined, and this fact represents an important feature of the K-distance. Some of the alternative metric geodesics between two probability functions are more interesting than the mixture, from the point of view of applications involving the idea of transport of masses. On this topic, we refer to the tutorial parts of \citet{santambrogio:2015OTAP}.   

For example, take the points 1, 2, 3 on the line. The mixture $\mu(t) = (1-t)\delta_1 + t \delta_3$, $t \in [0,1]$, moves the mass from 1 to 3 without passing through 2, which could be considered unnatural in some applications based on particle modeling. However, one can join the mixture of $\delta_1$ and $\delta_2$ to the mixture of $\delta_2$ and $\delta_3$ to get a metric geodesic whose middle point is $\delta_2$. The joining of two metric geodesic is a metric geodesic because $d(1,3) = d(1,2) + d(2,3)$.

\subsection{Content of the paper}  \Cref{sec:lipshitz} begins with an illustration of  Kantorovich-Rubinstein (1958) duality theory for the LP problem  \eqref{def:K-distance}. More details can be found, for example, in \S~1.2 and \S~3.1.1 of \citet{santambrogio:2015OTAP}. The duality theory teaches us that the K-distance obtains as a maximization problem on the class of Lipschitz functions. This entails that the distance induced on probability functions by the solution of the LP happens to be the restriction of a norm to the difference $\xi = \mu-\nu$ (see Kantorovich-Berstein norm in \cref{KBnorm}).

After these classic arguments, we deviate a bit and we introduce two other norms on space $M_0(X)$. The first one goes back to  \citet{arens|eells:1956}. We refer to the monograph \citet{weaver:2018-LA-2nd-ed}  for an exhaustive study about this subject. Though the $\ArensEells$-norm turns out to be equivalent to KB-norm, it is sometimes more tractable. For instance, this norm permits to establish---see  \Cref{prop:Vague}---nice properties for the functions $a(x,y)$ which are a surrogate of couplings $\gamma(x,y)$ in $\ArensEells$ spaces.

For the special class of finite metric spaces $(X,d)$ whose metrics lie in the cone  $\operatorname{CUT}(X)$, i.e., the $\ell_1$-embeddable metrics, we offer another natural extension to the space $M_0(X)$, through the construction of a C-norm. This will be discussed in \Cref{sub:CUTS}. Such a new norm sheds light on the nature of K-distance  for a class of $\ell_1$-embeddable metrics. 

In a good deal of the present paper the underlying space $X$ is endowed with the so-called graphic metric, defined for a weighted (undirected) graph. Namely: the sum of weights along a path defines its length and the distance between two points is the length of the shortest path linking them. The availability of a graph defining the metric is not a restrictive assumption. In fact, all distances on a finite set can be realized by (many) weighted graphs. On the topic of distances on finite sets, the reader is referred to the monograph \citet{deza|laurent:1997}. However, the assumption of a specific weighted graph is useful. First, the graph could be assigned by the specific application and, intuitively, one expects the optimal transport to flow along the ground ``geodesics''. Second, the explicit form of the K-distance is expressed, in some cases, through the given ground graph. Some useful facts about distances and about graphs are gathered in \Cref{sec:graphs} below.

A closed-form expression for the K-distance has been known even before the formalization by Kantorovich in the case of the distance induced by a total order, see \citet{salvemini:1939}. In such a case, the K-distance reduces to the weighted $L^1$-distance between the cumulative probability functions. This result has been recently extended to arbitrary finite trees in Th.~1 by  \citet{mendivil:2017} who provided a generalized closed-form expression of a K-distance. In the present paper (see \Cref{sec:trees}) we re-derive his result by means of other methods of proofs. Moreover, we show that this is a special case of the general theory for cut distances. The closed-form equation has interests in its own, as we show with a couple of examples. K-distance on trees has also been considered in \citet{kloeckner:2015}, and by M.~Sommerfeld and \citet{sommerfeld|munk:2018}. 

The existence of closed-form solutions for trees prompts for an inquiry about the use of spanning trees to compute the K-distance for general graphs. In fact, the distance turns out to be the minimum distance among the spanning trees of the assigned graph. This is discussed in \Cref{sec:extensions} and should be compared with Th.~2 of \citet{mendivil:2017} and with \citet{cabrelli|molter:1995}.

\subsection{Distances and graphs}
\label{sec:graphs}

A weight $w$ on the finite set $X$ is a symmetric mapping from  $X \times X$ to non-negative reals. The support of a weight defines a (undirected, simple) graph $G = (X,\mathcal E)$ with edges $\mathcal E = \setof{\set{x,y}}{w(x,y)>0}$. We write also $V(G) = X$ and $E(G)=\mathcal E$. A (un-directed, simple) graph without loops i.e., edges of the form $\set{x,x}$, can be seen as a weighted graph in which all the weights on edges are equal to 1. 

A path between $x$ and $y$ is a sequence of at least two vertices $x=x_0,\dots,x_n=y$ such that $x_{i-1}x_{i}$ is an edge, $i=1,\dots,n$.  A cycle is a path with $x_0=x_n$. The length of the path is $\sum_{i=1}^n w(x_{i-1},x_i)$. All graphs we consider are connected, that is, there is always a path connecting $x$ to $y$.

Given a weighted graph, the distance $d(x,y)$, of two vertices, is defined to be the length of the shortest path connecting $x$ and $y$. It is easy to see that it is actually a distance. If confusion can arise we will also write $d_G$ or $d_{G,w}$.

Conversely, if $d$ is a distance on $X$, then $d$ is a weight on the complete graph and the distance induced by the weight is equal to $d$. However, a smaller graph could define the same distance. In fact, giving a weighted graph is more than giving the mere distance. Other tools related to distances on finite sets will be introduced later.

A graph without cycles is called a tree if it is connected, a forest otherwise. In a tree, the length of a two-point path equals its weight. It is frequently useful to select a distinguished vertex, the so-called root of the tree, and to provide each edge with a direction in such a way that the distance from the root vertex increases in that direction. In such a case, we write the edge as $x \to y$. Given a vertex $x$, the set of all $y$ such that $x \to y$ (the children of $x$) is denoted $\chof x$. The unique parent of a non-root vertex $x$ is denoted by $x^+$. The partial order induced by a rooted tree is denoted by $x \preceq y$ or $y \succeq x$.

Recall that it is called a \textit{tree metric} (or \textit{tree-like}) a space which is isometrically embedded into a weighted tree. Under the so-called  \textit{four-point condition}, a metric space can be realized by a subset of a weighted tree.  See \citet{buneman:1974}.  Ultrametric spaces  fall into this class of metrics. We refer to the monograph by \citet{bollobas:1998} and \citet{deza|laurent:1997} for further results on these subjects.

\section{Duality, Arens-Eells spaces and CUT seminorms}\label{sec:lipshitz} 

A real function $u$ on $X$ is called 1-Lipschitz for the distance $d$, if $\aval{u(x)-u(y)} \leq d(x,y)$, for all $x,y \in X$. Equivalently, $u(y) \leq d(x,y) + u(x)$ for all $x,y \in X$.

The condition $\aval{u(x) - u(y)} \leq Kd(x,y)$ for some $K$ gives rise to the linear space of Lipschitz functions $\Lipschitz(d)$. The best constant $K$ is a semi-norm, $\normat {\Lipschitz(d)}{u}$. The set of 1-Lipschitz functions will denoted by $\Lipschitz_1(d)$. When the distance is generated by a weighted graph, it is enough to check the Lipschitz condition on edges, as established in the next proposition.

\begin{proposition}\label{prop:lipshitz}
Let $(X,w)$ be a weighted graph with associated distance $d$. A function $u \colon X \to \reals$ belongs to $\Lipschitz_1(d)$ if, and only if, $\aval{u(x) - u(y)} \leq d(x,y)$ for each edge $xy \in \mathcal E$.
\end{proposition}
More generally, we can say that $\aval{u(x)-u(y)} \leq K d(x,y)$ holds for each edge $xy \in \mathcal E$ if, and only if, $\normat {\Lipschitz(d)}{u} \leq K$.

Kantorovich duality theorem below is an application of LP duality. For a detailed treatment, see, for example, \citet{santambrogio:2015OTAP} and \citet{villani:2008optimal}.

\begin{theorem}[Kantorovich duality]\label{th:Kantorovich-Rubinstein}
\label{Kantorovich}
Let $\mu$ and $\nu$ be given probability functions on the finite metric space $(X,d)$ and $\mathcal P(\mu,\nu)$ be the set of couplings. Then, 
\begin{multline*}
 d(\mu,\nu) = 
 \min \setof{\sum_{x,y\in X} d(x,y) \gamma(x,y)}{\gamma \in \mathcal P(\mu,\nu)} = \\ \max \setof {\sum_{z \in X} u(z) (\mu(z)-\nu(z))}{u \in \Lipschitz_1(d)} \ .
\end{multline*}
\end{theorem}

The second term of the equality in \Cref{Kantorovich} shows that the distance $d(\mu,\nu)$ depends only on the difference $\xi = \mu - \nu$. For such functions we have $\sum_z \xi(z) = 0$ and $\sum_z \aval{\xi(z)} \leq 2$. Conversely, every function $\xi$ that satisfies $\sum_z \xi(z)=0$ is the difference of two probability functions if $\sum_z \aval{\xi(z)} \leq 2$. In fact, the assumptions made on $\xi$ imply $\sum _z \xi^+(z) = \sum_z \xi^-(z) \leq 1$. If the strict inequality holds, given any probability function $p$ we can choose $\alpha > 0$ such that both $\xi^+ + \alpha p$ and $\xi^- + \alpha p$ are probability functions.

If we ignore this restriction on the elements $\xi$, we obtain the vector space $M_0 (X)$ of zero-mass measure functions and we can hence define on this space the so-called Kantorovich-Bernstein norm (KB-norm),
\begin{equation}
\label{KBnorm}
  \normat{\KB}{\xi}  = \sup_{u \in \Lipschitz_1(d)} \sum_{z \in X} \xi(z) u(z) \ ,
\end{equation}
so that the K-distance is just the restriction of the KB-norm, i.e., $d(\mu,\nu) = \normat{\KB}{\mu-\nu}$.

\subsection{ Arens-Eells norm}
\label{sec:arens-eells-norm}

Let $\Lipschitz^+ (d)$ be the space of the Lipschitz functions  defined on the pointed metric space $(X,d,x_0)$, namely, the set of Lipschitz functions $u$ for which $u(x_0) = 0$ and where $x_0$ is a distinguished element of $X$. It turns out to be a Banach space by norm $\normof u_{\Lipschitz(d)}$, given by the smallest Lipschitz constant of $u$. The symbol $\Lipschitz_1^ +(d)$ denotes the unit ball and $\Extreme \Lipschitz_1^+(d)$ the set of its extreme points.

Each difference of delta functions belongs to $M_0(X)$ and the isometric property is verified through the dual formulation
\begin{multline}
\label{eq:ISOM}
  d(\delta_x,\delta_y) = \normat{\KB}{\delta_x - \delta_y} = \sup_{u \in \Lipschitz^+_1 (d)} \scalarof {\delta_x - \delta_y} u = \\ \sup_{u \in \Lipschitz^+_1(d)} (u(x) - u(y)) = d(x,y) \ .
\end{multline}

The difference of delta functions spans the whole space, i.e., every $\xi \in M_0(X)$ can be written as
\begin{equation}\label{eq:RAP}
  \xi = \sum_{x,y} a(x,y) (\delta_x - \delta_y) \ , \quad A = [a(x,y)]_{x,y \in X} \in \reals^{X \times X} \ .
\end{equation}

If we compute the KB-norm of a generic function $\xi$ in \cref{eq:RAP}, we get through \cref{eq:ISOM}
\begin{equation*}
  \normat{\KB}{\xi}  = \normat{\KB}{\sum_{x,y} a(x,y) (\delta_x - \delta_y)} \leq \sum_{x,y\in X} \aval{a(x,y)} d(x,y) \ .
\end{equation*}
The vector space $M_0(X)$ can be endowed with the following norm, in which case it will be called the Arens-Eells space $\AEspace X$. Here, we follow the presentation by \citet{weaver:2018-LA-2nd-ed}.

\begin{definition}\label{def:AEspace}
The norm $\normat {\ArensEells} {\xi}$ is defined by
\begin{equation}\label{eq:AEnorm}
\normat {\ArensEells} {\xi} = \inf \set{\sum_{x,y \in X} \aval {  a(x,y)} d(x,y) } \ ,
\end{equation}
where the inf is made on all the representations of $\xi$ in \cref{eq:RAP}.
\end{definition}

Arens-Eells construction has a wider range of application than finite metric spaces. In fact, in a more general setting,  Arens-Eells space is defined as the norm-closure of the space of all zero-mass measures with finite support on an arbitrary metric space $X$. It is also known in literature as a Lipschitz-free space over $X$ and frequently denoted by $\mathcal{F}(X)$.
  
  The Banach space $\AEspace X$ is a predual of $\Lipschitz^ +(d)$.  More specifically, the linear isometry $T: \AEspace X^* \rightarrow \Lipschitz^ +(d)$ is given by
    \begin{equation*}
  T(\phi) (x) = \phi(\delta_x - \delta_{x_0})
  \end{equation*}
  for every $\phi \in \AEspace X^*$,  $x \in X$ and where $x_0$ is the distinguished point of $X$.

Consequently,
\begin{equation}
\label{eq:DUAL}
\normat {\ArensEells} {\xi} =  \normat{\KB}{\xi} = \sup \set {\scalarof {\xi} {u} \mid u \in \Lipschitz_1^ +(d)}
\end{equation}
for all $\xi \in \AEspace X$. In addition, there exists a multi-mapping $J \colon \AEspace X \to \Lipschitz_1^ +(d)$ such that the alignment condition $\scalarof {\xi} {J(\xi)} = \normat {\ArensEells} {\xi}$ is satisfied. In our finite-dimensional setting, it holds  also the reflexivity property, $\AEspace X = \Lipschitz^+(X)^*$. 

Here, we just recall an important property of the Arens-Eells norm (see \citet{weaver:2018-LA-2nd-ed} for more details). 

\begin{proposition}
\label{prop:LARG}
The norm $\normat {\ArensEells} {\cdot}$ is the largest semi-norm on the space $\AEspace X$ which satisfies $\normof {\delta_x - \delta_y} \leq d(x,y)$. 

\begin{proof}
  If $\normof \cdot$ is a semi-norm that satisfies the claimed requirements, we have from \cref{eq:RAP} that
\begin{equation*}
\normof \xi = \normof {\sum_{x,y \in X} a(x,y) (\delta _x - \delta_y)} \leq \sum_{x,y \in X} \aval {  a(x,y)} d(x,y)   
\end{equation*}
is true for any representation of $\xi$ as a linear combination of differences of Dirac functions. Consequently, $\normof \xi \leq \normat {\ArensEells} {\xi}$.
\qed \end{proof}
\end{proposition}
Another important property is that if $X_0$ is a nonempty subset of the metric space $X$, then the identity map takes $\AEspace{X_0}$ isometrically into $\AEspace{X}$, see Theorem 3.7 in \citet{weaver:2018-LA-2nd-ed}. 

\subsection{The cut seminorms}
\label{sub:CUTS}
We endow $M_0(X)$ by another (semi-)norm via the notion of cut semi-metrics defined below, see \citet{deza|laurent:1997}. A semi-metric (or pseudo-metric) is a relation on $X$ with all the properties of a distance but separation of points.

\begin{definition}
\label{def:CUT}
Let $S$ be a subset of a given finite set $X$. The relation $\delta_S$ defined in $X$ by
\begin{equation*}
 \delta_S(x,y) = 1 \quad \text{if}\quad \#[S \cap \set{x,y}] = 1 \quad \text{and} \quad  \delta_S(x,y) = 0 \quad \text{otherwise},
\end{equation*}
is a semi-metric called \emph{cut semi-metric}. The \emph{cut cone} $\CUT(X)$ is the set of distances in $X$ which lie in the cone generated by the cut semi-metrics, namely,
\begin{equation*}
\CUT(X) = \setof{d = \sum_{S \subseteq X}  \lambda_S \delta_S}{\lambda_S \geq 0,\quad \forall S \subseteq X} \ .
\end{equation*}
\end{definition}

A decomposition $d = \sum_{S \subseteq X} \lambda_S \delta_S$ will be said to be $x_0$-adapted if $x_0 \in S \Rightarrow \lambda_S = 0$.  Since $\delta_S = \delta_{\bar S}$, one can always turn a decomposition into an adapted one for some distinguished point $x_0 \in X$.

A key property of the cut cone is that $d \in \CUT(X)$ if and only if the associated metric space $(X,d)$ is $\ell_1$-embeddable, see Prop. 4.2.2 \citet{deza|laurent:1997}.

A representation $d = \sum_{S \subseteq X} \lambda_S \delta_S$ of a given metric is said to be a realization of $d$. In general, the realization of a given distance $d$ is not unique. A metric is called $\ell_1$-\emph{rigid} if its realization is unique.

In view of \Cref{def:CUT}, let us introduce the following family of semi-norms. 
\begin{definition}\label{def:cut-norm}
Given a finite set $X$ and a family of non negative scalar numbers  $C = \set{\lambda_S}_{S \subseteq X}$, set
\begin{equation*}
\normat{C}{\xi} = \sum_{S \subseteq X} \lambda_S \aval{\sum_{x \in S} \xi(x)}, \quad \forall \xi \in M_0(X) \ .
\end{equation*}
 \end{definition}
\begin{proposition}
The semi-norm $\normat{C}{\cdot}$ extends the semi-metric $d = \sum_{S \subseteq X} \lambda_S \delta_S$ contained in  $\CUT(X)$. Moreover, $\normat{C}{\xi} \leq \normat {\ArensEells} {\xi}$.
\end{proposition}
\begin{proof}
Clearly any $\normat{C}{\cdot}$ is a semi-norm on $M_0(X)$. To show the second claim, it is enough to check that the semi-norms $\aval{\sum_{x \in S} \xi(x)}$ extend the cut semi-metrics $\delta_S$, for all $S \subseteq X$.
In fact, for all $y,z \in X$, we have
\begin{equation*}
\aval{\sum_{x \in S} [\delta_y(x) -  \delta_z(x)]} = \delta_S (y,z) \ .    
\end{equation*}
 \Cref{prop:LARG} provides the last claim.
\qed \end{proof}

It will be discussed in \Cref{subsect: tree-like} that for trees, and more generally for tree-like spaces, the Arens-Eells norm is exactly the cut norm for a special family $C$, i.e.,  $\normat{C}{\xi} = \normat {\ArensEells} {\xi}$.

Observe further that non-$\ell_1$-rigid metrics give rise to many different cut norms. We present below the classical Ex.~4.3.7 of \citet{deza|laurent:1997}. 

\begin{example}
\label{ex:1}
It is well-known that the discrete metric $(X,d)$ is $\ell_1$-rigid if and only if $\#X < 4$. For $\#X = n \geq 4$, the discrete metric admits several  distinct realizations. For instance,
\begin{equation*}
d= \frac{1}{2}\sum_{x \in X} \delta_{\set{x}} = \frac{1}{2(n-2)}\sum_{x \neq y } \delta_{\set{x,y}}
\end{equation*}
which generate, respectively, the two C-norms
\begin{equation*}
\normat{C_1}{\xi} = \frac{1}{2}\sum_{x \in X} \aval{\xi(x)} \ , \quad \normat{C_2}{\xi} = \frac{1}{2(n-2)}\sum_{x \neq y } \aval{\xi(x) + \xi(y)} \ .
\end{equation*}
As will be seen later we have $\normat{C_1}{\xi} = \normat {\ArensEells} {\xi}$, and so $\normat{C_2}{\xi} \leq \normat {\ArensEells} {\xi}$.  
\end{example}

\section{Distance induced by a graph}
\label{sec:distance-graph}
We commence by giving a useful notion that refines the adjacency property for points of a graph.
\begin{definition}\label{def:close}
Two vertices $x,y$ of a weighted graph $(X,w)$ are said to be \emph{close} if they are adjacent and, in addition, $d(x,y) = w(x,y)$, i.e., the path $xy$ is one of the shortest paths joining the points themselves.
\end{definition} 

Adjacent vertices are necessarily close in a tree. So too are all the adjacent pairs in an unweighted graph. Observe further that, in any path $x_1,x_2,\dots,x_n$ of minimum length, two adjacent vertices are necessarily close. This is the reason why it holds the equality
 \begin{equation}
 \label{eq:CZZ}
    d(x_1,x_n) = \sum_{i=1}^{n-1} d(x_i,x_{i+1})
  \end{equation}
along points of a path of minimal length.
  
To see this, suppose not. We would have $w(x_i,x_{i+1}) \geq d(x_i,x_{i+1})$ for all $i$ and $w(x_j,x_{j+1}) > d(x_j,x_{j+1})$ for some $j$. Therefore
  \begin{equation*}
    d(x_1,x_n) = \sum_{i=1}^{n-1} w(x_i,x_{i+1}) > \sum_{i=1}^{n-1} d(x_i,x_{i+1}) \ 
  \end{equation*}
  which contradicts the triangular inequality.
  
It is worth remarking that one could replace adjacent points with close points in \Cref{prop:lipshitz} too.
  
\subsection{Extreme points}
\label{sec:extr-points-lipsh}

It is known in literature a characterization of the extreme points of the unit ball of the normed space $\Lipschitz^+(d)$, for generic metric spaces, cf. \citet{farmer:1994}, \citet{smarzewski:1997}.

Here, we are concerned with a useful qualification that holds in finite spaces. For ease of the reader we provide a complete proof which essentially follows Th.~2.59 \citet{weaver:2018-LA-2nd-ed}.

\begin{theorem}
\label{th:WEA}
Let $(X,x_0)$ be a pointed finite metric space. A function $f \in \operatorname {Lip}_1^+(d)$ is extremal if and only if for every $x \in X$ there is a path $x_0,x_1,\dots,x_{n-1}$, with $x_{n-1}=x$, such that
\begin{equation*}
\aval{f(x_i) - f(x_{i-1})}
= d(x_i,x_{i-1})
\end{equation*}
for $i = 1,..,n-1$.
When the distance is induced by a graph, the path linking $x_0$ and $x$ can be taken to be a sequence of close points.
\end{theorem}

\begin{proof}
Suppose a function $f$ satisfies the stated condition and consider the functions $f \pm u \in \Lipschitz_1^+(d)$. We must show that $u = 0$. 

Fixing $x \in X$, by hypothesis there exists a path $x_0,x_1, \dots,x_{n-1} = x$, with \\ $\aval{f(x_i) - f(x_{i-1})}
= d(x_i,x_{i-1})$.

Further, in view of \cref{prop:lipshitz}, it holds
\begin{equation*}
\aval{f(x_i) - f(x_{i-1}) + u(x_i) - u(x_{i-1})}
\leq d(x_i,x_{i-1})    
\end{equation*}
as well as
\begin{equation*}
\aval{f(x_i) - f(x_{i-1}) - u(x_i) + u(x_{i-1})}
\leq d(x_i,x_{i-1}).    
\end{equation*}
Fixing the index $i$ and setting, for short,  $ a = f(x_i)-f(x_{i-1})$, $ h = u(x_i)-u(x_{i-1})$ and $d=d(x_i,x_{i-1})$, we get the three conditions
\begin{equation*}
\aval {a} =d,\ \aval {a+h} \leq d,\  \aval {a-h} \leq d
\end{equation*}
which imply necessarily $ h = 0$.

Since $u(x_0) =0$ it follows that $u$ vanishes along that path and so $u(x) = 0$. In turn, this implies $u(x) = 0$ for each $x \in X$, as desired.

As far as it concerns the necessity condition, we shall treat the case where the finite space is a graph. Assume that the condition stated fails for some $\bar x$ and for every path $x_0,x_1,\dots, x_{n-1} = \bar x$ for which the points $x_i,x_{i+1}$ are close.

Define the following function $u: X \rightarrow \reals$
\begin{equation*}
u(x) = \min \left[ \sum_{i=1}^{m-1} d(z_i,z_{i-1}) - \sum_{i=1}^{m-1} \aval{f(z_i) - f(z_{i-1})} \right]   \ . 
\end{equation*}
where the minimum is taken over all the sequences of close vertices from $x_0$ to $x$.

Clearly, $u(\bar x) >0$ in that, by construction,
\begin{equation*}
\sum_{i=1}^{n-1} \aval{f(x_i) - f(x_{i-1})}  < \sum_{i=1}^{n-1} d(x_i,x_{i-1}) \ 
\end{equation*}
holds for all the paths linking $x_0$ and $\bar x$. 

Take now any pair $x,y \in X$ of close points. Any sequence linking $x_0$ and $x$ can be extended to a sequence linking $x_0$ and $y$, by adding the additional point $y$ preserving the property of being a sequence of close points. 

It follows 
\begin{equation*}
u(y) \leq u(x) + d(x,y) - \aval{f(x) - f(y)} \     
\end{equation*}
for each pair of close vertices.
Switching $x$ and $y$, we get by some algebra
\begin{equation*}
\aval{u(y) - u(x)} + \aval{f(x) - f(y)} \leq d(x,y)   
\end{equation*}
that in turn implies the two functions $f\pm u$ are 1-Lipschitz for close points. As already discussed this entails that $f\pm u \in \operatorname {Lip}_1^+(d)$ with $u \neq 0$, which is a contradiction because $f$ was assumed to be extremal.
\qed \end{proof}

Let us look at some specific classes of graphs. In the first one, we are dealing with a straightforward application that  does not require a further proof.
\begin{proposition}
\label{prop:lipextreme}
In a weighted tree, a function $u$ is extremal in the unit ball of ${Lip}_1^+(d)$ if, and only if, $\aval {u(x)-u(y)} = d(x,y)$ for each pair of adjacent vertices.
\end{proposition}

Next consider a set X equipped by the discrete distance. In another words, $X = K_n$ is the unweighted complete graph. Observe that the distance admits the realization $d = \frac{1}{2} \sum_{x \in X} \delta_{\set{x}}$.

\begin{proposition}
\label{prop:GGG}
Let $d$ be the discrete distance. The function $u \in \Extreme \operatorname {Lip}_1^+(d)$ if, and only if, $u = \pm I_Y$, where $I_Y$ is the indicator function of a nonempty subset $Y \subseteq X \setminus{x_0} $.
\end{proposition}

\begin{proof}
The functions $\pm I_Y$ are surely extremal. Actually, if we pick $x \in Y$, the path $x_0x$ satisfies the sufficient conditions of \Cref{th:WEA}. While, if $x \notin Y$, the path $x_0x_1x$ does, where $x_1$ is any point of $Y$.

To show that there are no others, we put in place the necessary conditions. If $x_0x_1,..x_n$ is any sequence claimed in \Cref{th:WEA}, then $x_1 = \pm 1$. Since $u \in \operatorname {Lip}_1^+(d)$, we infer that either $0 \leq u(x) \leq u(x_1) = 1$ or $-1 \leq u(x) \leq 0$.

Consider the positive case, the other is similar. Suppose by contradiction that the function takes at least three values. Hence
\begin{equation*}
    0= u(x_0) < u(\bar x) < u(x_1)=1 \ .
\end{equation*}
holds for some $\bar x \in  X$. But then for any path linking $x_0$ and $\bar x$ the necessary condition of \Cref{th:WEA} fails. 
\qed \end{proof}

\begin{remark}
\label{remark:SSS}
A generalization of the previous result gives the extreme points for the metric space $(X,d)$ where $d(x,y) = \phi(x) + \phi(y)$, for $x \neq y$, being $\phi: X \rightarrow \reals$ a fixed strictly positive function. $K_n$ is just the special case $\phi = 1/2$.

An almost identical proof to \Cref{prop:GGG} leads to the extremal functions $\pm f$, given by
\begin{equation*}
f(x) = \phi(x) + \phi(x_0) \ \ \text{if }x \in Y ,\quad f(x) = -\phi(x) + \phi(x_0) \ \  \text{if}  \ x \in \Bar{Y}\setminus{x_0} 
\end{equation*}
and $Y \subseteq X \setminus{x_0} $ is any nonempty set. 
\end{remark}
Let us calculate the K-distance  for this last example. 

\begin{example}  If $(X,\phi)$ is the metric space of \Cref{remark:SSS}, then 
\begin{equation}\label{XXX}
\normat {\ArensEells} {\xi} = \sum_{x \in X} \phi(x) \aval{\xi(x)}.
\end{equation}
Actually, by \cref{eq:DUAL} we have to seek for an extremal function $\pm f$ that maximizes 
$\sum_{x \in X} f(x) \xi(x)$.

For the extremal functions $f$, we must maximize the amount 
\begin{equation*}
\sum_{x \in Y} \xi(x) \phi(x) - \sum_{x \in \Bar{Y}\setminus{x_0}} \xi(x)\phi(x) -\xi(x_0)\phi(x_0)
\end{equation*}
while, if the functions is $-f$ we have
\begin{equation*}
-\sum_{x \in Y} \xi(x) \phi(x) + \sum_{x \in \Bar{Y}\setminus{x_0}} \xi(x)\phi(x) +\xi(x_0)\phi(x_0).
\end{equation*}
Hence the maximum value will be given by $\sum_{x \in X} \phi(x) \aval{\xi(x)}$.

\end{example}
This metric admits the realization $d= \sum_{x \in X} \phi(x)  \delta_{\set{x}}$ and it also worth remarking that it is a tree metric. In fact, $(X,d)$ can be embedded into the star graph $V=X \cup \set{r}$ having edges $(x,r)$ and weights $w(x,r)=\phi(x)$. So $X$ agrees with the leaves of this tree.

\begin{example}[Linear order]
\label{ex:lin}
Let the graph be the total ordering $1 \rightarrow 2 \rightarrow \cdots \rightarrow n $, with positive weights $w_{i,i+1} = d(i,i+1) = d_i$.
By \Cref{prop:lipextreme}, the extreme points of the unit ball $\Lipschitz_{1}^+( d ) $ are the functions $u$ satisfying $\aval{u_{i}-u_{i+1}} = d_i $, for $i= 1,2,.., n-1$. It is not difficult to get that
\begin{equation*}
\normat{\ArensEells} {\xi} = \sum_{i=1}^{n-1} d_i \left\vert \Xi(i) \right\vert .
\end {equation*}
where $\Xi(i) = \xi _1 +\xi _2+\dots+ \xi _{i} $.
\end{example}

In this last example the K-distance reduces to the ordinary distance between the two cumulate functions. This is of some interest in the present paper, because it will be generalized to trees in \Cref{sec:trees}.

One of the tools useful to study the case of trees relies on the detection of the extreme points of the unit ball of $\operatorname {Lip}^+(d)$, as described in the next proposition.

\begin{proposition}
\label{prop:exttree}
Let $T=(X,w)$ be a rooted tree. A function $u \in \Extreme \operatorname {Lip}_1^+(d)$ if and only if it is of the type
\begin{equation}
\label{eq:LIPf}
u_\epsilon(x) = \sum_{y \preceq x} d(y,y^+) \epsilon(y) \ ,
\end{equation}
for all $x \in X \setminus{x_0}$, and $u_\epsilon(x_0) = 0$ otherwise, where $\epsilon$ is any given $\epsilon: X \setminus{x_0} \rightarrow \set{-1,1}$.
\end{proposition}

\begin{proof}
It suffices to remark that the functions $u_\epsilon$ are recursively generated by the equation
\begin{equation}
\label{eq:RECU}
u_\epsilon(y) = u_\epsilon(x) + d(x,y)\epsilon (y), \quad \forall y \in \chof x
\end{equation}
with initial condition $u(x_0) = 0$. The desired result is a consequence of \Cref{prop:lipextreme}.  
\qed \end{proof}

\subsection{A support property}

Many coefficients $a(x,y)$ of the linear combination in \cref{eq:RAP} are, in fact, not needed. They can be avoided if the purpose is to compute the inf of \cref{eq:AEnorm}. The following theorem is a key result that will be repeatedly used in the following sections.

\begin{theorem}
  \label{prop:Vague}
Assume that the distance $d$ is generated by a weighted graph. The class of functions $a(x,y)$, employed in \cref{eq:AEnorm}, can be restricted to the one satisfying the following two conditions:

\begin{enumerate}[label = \roman*)]
\item\label{item:vague-2} if $a(x,y) \neq 0$, then $x$ and $y$ are close.
\item\label{item:vague-1} the graph $\setof{(x,y)}{a(x,y) \neq 0}$ has no cycle.
\end{enumerate}
\end{theorem}

The substantive difference between the two above properties is that condition (ii) pertains metric spaces while the graph structure is not required. Indeed, this cycle-free property holds in general metric spaces. This has been remarked in \cite[Prop. 24]{bergman:2008}. Analogous results are given by \cite[Prop. 7.2] {barvinok:2002}.   

\begin{proof} \emph{\Cref{item:vague-2}} 
Let us first assume that in \cref{eq:RAP} there is a non-zero term $a(x,y) (\delta_x - \delta_y)$, where $x$ and $y$ are not adjacent. Let $x_1,x_2,\dots,x_n$ be a geodesic path joining $x =x_1$ to $y = x_n$. By \cref{eq:CZZ}, 
  \begin{equation*}
    d(x_1,x_n) =  \sum_{i=1}^{n-1} d(x_i,x_{i+1}) \quad \text{and} \quad \delta_x - \delta_y = \sum_{i=1}^{n-1} (\delta_{x_i} - \delta_{x_{i+1}})
  \end{equation*}
Therefore, if the addendum $a(x,y) (\delta_x - \delta_y)$ is replaced by

$a(x,y) \sum_{i=1}^{n-1} (\delta_{x_i} - \delta_{x_{i+1}})$, 
the contribution to the norm will remain the same, since
\begin{equation*}
\aval {a(x,y)} \sum_{i=1}^{n-1} d(x_i,x_{i+1})= \aval {a(x,y)} d(x,y).     
\end{equation*}

Consequently, the term $a(x,y)(\delta_x - \delta_y)$ may be removed, whenever $x$ and $y$ are not adjacent.

Suppose now that $x$ and $y$ are adjacent but not close. That means that a geodesic path $x_1,x_2,\dots,x_n$ exists with $x =x_1$ and $y = x_n$, and the strict inequality $d(x,y) > \sum_{i=1}^{n-1} d(x_i,x_{i+1})$ holds. In this case,
\begin{equation*}
\aval {a(x,y)} \sum_{i=1}^{n-1} d(x_i,x_{i+1}) < \aval {a(x,y)} d(x,y) \ .     
\end{equation*}
Once again the term may be removed, if the pair of vertices is not close. 

\emph{\Cref{item:vague-1}}
The argument will unfold along the following lines. Let
\begin{equation*}
\xi = \sum_{x,y \in X} \widetilde a(x,y) (\delta _x - \delta_y) 
\end{equation*}
be an optimal representation of $\xi$. That is, let
$\normat {\ArensEells} {\xi} = \sum_{x,y \in X} \aval {  \widetilde a(x,y)} d(x,y)$. 

In addition, let us suppose that it is a minimal, i.e, it contains the minimum number of non-vanishing coefficients $a(x,y)$. A minimal representation does exist but clearly it is not a unique one. For instance, as $a(x,y)(\delta_x - \delta_y) = -a(x,y)(\delta_y - \delta_x)$, any change of signs for the coefficients produces another optimal minimal representation.

Suppose by contradiction  that in a minimal representation of $\xi$  there is a set of non-zero coefficients $\widetilde a(x,y)$, $(x,y) \in \mathcal S$, whose graph $(S,\mathcal S)$ is a  cycle. We can write
\begin{equation*}
 \xi = \sum_{(x,y) \in \mathcal S} \widetilde a(x,y)(\delta_{x} - \delta_{y}) + A,  
\end{equation*}
where $A$ includes all the other remaining terms.

Moreover, by arranging signs of coefficients, we can suppose that the cycle is directed, so that we have $\sum_{(x,y) \in \mathcal S} (\delta_{x} - \delta_{y}) = 0$, 
then 
\begin{equation*}
 \xi = \sum_{(x,y) \in \mathcal S}[\widetilde a(x,y) - t] (\delta_{x} - \delta_{y}) + A 
\end{equation*}
holds for any scalar $t$. It follows that 
\begin{equation*}
\normat {\ArensEells} {\xi} = \inf_{t \in \reals} \sum_{(x,y) \in \mathcal S} \aval {  \widetilde a(x,y) - t} d(x,y) + A  \ .
\end{equation*}

On the other hand, the scalar function
\begin{equation*}
 t \mapsto  \sum_{(x,y) \in \mathcal S} \aval {  \widetilde a(x,y) - t} d(x,y)
\end{equation*}
is convex and piece-wise linear. Consequently, it attains its minimum value at some point $t = \widetilde a(\bar x,\bar y)$, with $(\bar x,\bar y) \in \mathcal{S}$. Hence it would be
\begin{equation*}
\normat {\ArensEells} {\xi} = \sum_{(x,y) \in \mathcal S \setminus{(\bar x, \bar y)}} \aval {  \widetilde a(x,y) - \widetilde a(\bar x, \bar y)} d(x,y) + A  \ 
\end{equation*}
but it contains a smaller number of non-zero coefficients, a contradiction.
\qed \end{proof}

\subsection{A decomposition property}\label{sec:decomp}

Under proper conditions,  the norm associated with a graph  may be deduced by decomposing the graph itself into a certain number of sub-graphs. 

Following \citet{weaver:2018-LA-2nd-ed}, if $\set{X_\lambda}$ is a family of pointed metric spaces, the sum $\coprod X_{\lambda} $ denotes their disjoint union with all base points identified and metric
\begin{equation*}
d(x,y) = d(x,e) + d(e,y) ,    
\end{equation*}
whenever $x$ and $y$ belong to distinct summands and $e$ is the common base point. This operation is also known as \textit{1-sum operation}, see \S~7.6 of \citet{deza|laurent:1997}.

A connected graph $G = (X,\mathcal E)$ is called \textit{decomposable} if the graph $G \setminus{x_0}$ is not connected for a certain vertex $x_0 \in X$.  If $G \setminus{x_0}$ has $k \geq 2$ components, then the set of vertices $X$ can be partitioned as $X=X_1 \cup X_2 \cup \cdots \cup X_k$, with $X_i \cap X_j = \set{x_0}$.
Consequently, if $x$ and $y$ lie into two distinct components $X_i$ and  $X_j$, then
\begin{equation*}
d(x,y) = d(x,x_0)+d(x_0,y).    
\end{equation*}
Hence, by means of vertex $x_0$, the graph $G = (X,\mathcal E)$ splits into $k$ sub-graphs $G_i = (X_i,\mathcal E_i)$, $i=1,\dots,k$, having in common the vertex $x_0$, and we can adopt the notation $G = \coprod_{i=1}^k G_i$, where 
$X = \coprod_{i=1}^k X_i$.

Observe further that any vector $\xi \in M_0(X)$ has a canonical decomposition $\xi = \xi^1 + \xi^2 + \cdots + \xi^k$, with $\xi^i \in M_0(X_i)$, and where
\begin{equation*}
\xi^i(x) = \xi (x)  \ \text{for}\ x \in X_i \setminus{x_0} \ \text{and} \ \xi^i(x_0) = -\sum_{x \in X_i \setminus{x_0}} \xi(x).     
\end{equation*}

After these preliminaries, we can state the following result, whose proof is referred to Prop.~3.9 in  \citet{weaver:2018-LA-2nd-ed}.

\begin{proposition}
Assume that a vertex $x_0 \in X$ splits the graph  $G = (X,\mathcal E)$ into $k$ components $G_i = (X_i,\mathcal E_i)$. Then
\begin{equation*}
 \AEspace {\coprod_{i=1}^k X_i} \cong \bigoplus_i  \AEspace {X_i}   ,   
\end{equation*}
with,
\begin{equation*}
\normat{G}{\xi} = \normat{G_1}{\xi^1} + \normat{G_2}{\xi^2} +\cdot + \normat{G_k}{\xi^k} \ .  
\end{equation*}
\end{proposition}

\begin{example}\label{example:2T}
In the two-cycles graph
$\vcenter{\hbox{\begin{tikzpicture}
   [scale=.15,auto=left]
  \node (n1) at (-8,0) {1};
  \node (n2) at (0,0)  {2};
  \node (n3) at (8,0)  {3};
  \node (n4) at (4,4) {4};
  \node (n5) at (-4,4)  {5};
  \foreach \from/\to in {n1/n2,n2/n3,n3/n4,n4/n2,n2/n5,n5/n1}
    \draw (\from) -- (\to);
  \end{tikzpicture}}}$, the vertex $2$ breaks it into $2$ components. In view of Example 1, the norm of an element $\xi= (\xi_1,\xi_2,\xi_3,\xi_4,\xi_5)$ will be given by
\begin{equation*}
\normat {\ArensEells}  {\xi} = \frac{1}{2}(\aval{\xi_1} + \aval{\xi_5} + \aval{\xi_1 + \xi_5} + \aval{\xi_4} + \aval{\xi_3} + \aval{\xi_4 +\xi_3}). 
\end{equation*}
Note incidentally that by what has been discussed in \Cref{sub:CUTS} the graph distance generated by the two-cycles graph is a CUT metric having the realization
\begin{equation*}
  d = \frac{1}{2}(\delta_{\set{1}} + \delta_{\set{5}} + \delta_{\set{1,5}} + \delta_{\set{4}} + \delta_{\set{3}} + \delta_{\set{3,4}} ). 
\end{equation*}
\end{example}

\section{Trees}\label{sec:trees}

This section is devoted to the specific analysis of the K-distance for weighted trees. 
Let $T=(X,w)$ be a rooted tree, where  $x_0$ denotes the root. Each vertex $x \in X$ can  be classified according to its depth, that is, its (un-weighted) distance from that root. Here we are going to utilize the order $\preceq$ defined in \Cref{sec:graphs}.

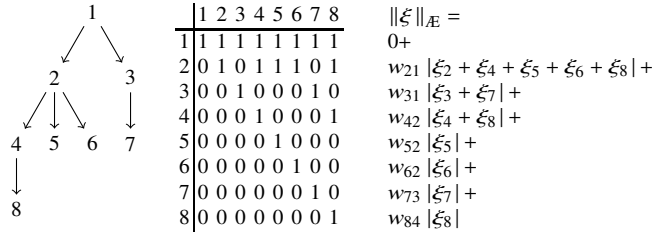
\begin{figure}\centering
  \begin{tabular}{lcr}
 \begin{tikzpicture}
  [scale=.25,auto=left]
  \node (n1) at (0,0) {1};
  \node (n2) at (-2,-3.5)  {2};
  \node (n3) at (2,-3.5)  {3};
  \node (n4) at (-4,-7) {4};
  \node (n5) at (-2,-7) {5};
  \node (n6) at (0,-7) {6};
  \node (n7) at (2,-7) {7};
  \node (n8) at (-4,-10.5) {8};
  \foreach \from/\to in {n1/n2,n1/n3,n2/n4,n2/n5,n2/n6,n3/n7,n4/n8}
    \draw[->] (\from) -- (\to);
\end{tikzpicture}
   &
\begin{array}[b]{c|cccccccc}
  & 1 & 2 & 3 & 4 & 5 & 6 & 7 & 8 \\
\hline
1 & 1 & 1 & 1 & 1 & 1 & 1 & 1 & 1 \\
2 & 0 & 1 & 0 & 1 & 1 & 1 & 0 & 1 \\
3 & 0 & 0 & 1 & 0 & 0 & 0 & 1 & 0 \\
4 & 0 & 0 & 0 & 1 & 0 & 0 & 0 & 1 \\
5 & 0 & 0 & 0 & 0 & 1 & 0 & 0 & 0 \\
6 & 0 & 0 & 0 & 0 & 0 & 1 & 0 & 0 \\
7 & 0 & 0 & 0 & 0 & 0 & 0 & 1 & 0 \\
8 & 0 & 0 & 0 & 0 & 0 & 0 & 0 & 1 \\
\end{array}
&
    \begin{array}[b]{rl}
      &\normat {\ArensEells}{\xi} = \\ &0 + \\
      &w_{21}\aval{\xi_2+\xi_4+\xi_5+\xi_6+\xi_8} + \\
      &w_{31}\aval{\xi_3+\xi_7} + \\
      &w_{42}\aval{\xi_4+\xi_8} + \\
      &w_{52}\aval{\xi_5} + \\
      &w_{62}\aval{\xi_6} + \\
      &w_{73}\aval{\xi_7} + \\
      &w_{84}\aval{\xi_8}
    \end{array}
  \end{tabular}
  \caption{\emph{Left panel:} A rooted tree in which each vertex is ordered by its distance from the root. \emph{Middle panel:} The adjacency matrix $E^*$ of the descendent relation. \emph{Right panel:} The Arens-Eells norm derived from the adjacency matrix $E^*$. \label{fig:rooted-tree}}
\end{figure}

Define the cumulative function defined by
\begin{equation*}
 \Xi(x) = \sum_{y \succeq x} \xi (y) \ ,   
\end{equation*}
 for $x \in X$ and $\xi \in M_0(X)$. Note that $\Xi(x_0) = \sum_{y \in X} \xi(y) = 0$. See also the right panel of \cref{fig:rooted-tree}.
Recall that $w(x,y) = d(x,y)$ holds for the edges $xy$ of a tree. Moreover, the shortest path is unique and it is the same under all weights and for all selections of a vertex as a root. 

\begin{theorem}
\label{th:tree}
Let $T = (X,w)$ be a weighted rooted tree. Then
\begin{equation}
\label{eq:DIS}
\normat {\ArensEells} {\xi}  = \sum_{x \in X \setminus{x_0}}  d(x,x^+) \aval {\Xi (x)} \ ,
\end{equation}
for every $\xi \in \AEspace X$. An equivalent expression of the norm is
\begin{equation}
\label{eq:MAT}
\normat {\ArensEells} {\xi}  = \sum_{y \in X}  \xi (y)  \sum_{x_0 \neq x \preceq y} d(x,x^+) \Varsig \Xi(x) \ , 
\end{equation}
where $\Varsig(\cdot) \in \set{-1,1}$ is a sign function, taking value at zero either  $\Varsig(0) = +1$ or $\Varsig(0) = -1$.
\end{theorem}

Formula \cref{eq:DIS} was given by \citet{mendivil:2017}. The interest of our presentation relies on the methods of proofs. Actually, we provide two different proofs for this theorem, each of them quite instructive in itself. The first one uses an algebraic argument based on the key result of \Cref{prop:Vague}. The alternative proof relies on the characterization of the extremal points of $\operatorname {Lip}_1^+(d)$.

\begin{proof}[First proof of \Cref{th:tree}] 
If the graph is a tree, \Cref{prop:Vague} implies that every $\xi \in \ArensEells (X)$ can be written as 

\begin{equation}
\label{eq:ZAZA}
\xi  = \sum_{x\in X \setminus{x_0}} a(x,x^+) (\delta_x - \delta_{x^+}) \ .
\end{equation}

In such a case, the above equation can be uniquely solved for the $a(x,x^+)$'s. Actually, from \cref{eq:ZAZA} we get
\begin{equation*}
\xi(z)  = a(z,z^+) - \sum_{x \in \chof z} a(x,z) 
\end{equation*}
for all $z \neq x_0$ and under the convention $\sum_{x \in \emptyset} a(x,z) = 0$.  Hence,
\begin{equation*}
 \Xi(x) = \sum_{z \succeq x} \xi (z) = \sum_{z \succeq x} a(z,z^+) - \sum_{z \succeq x} \sum_{x \in \chof z} a(x,z) = a(x,x^+) \ .
\end{equation*}
From this equality and the definition of the Arens-Eells norm,  \cref{eq:DIS} follows. 

With regard to \cref{eq:MAT}, it suffices to  interchange the order of two summations in \cref{eq:DIS}. More precisely,
\begin{multline*}
\normat {\ArensEells} {\xi}  = \sum_{x \in X \setminus{x_0}} d(x,x^+) \Varsig (\Xi(x)) \sum_{y \succeq x} \xi (y) = \\ \sum_{x \in X \setminus{x_0}} d(x,x^+) \Varsig (\Xi(x)) \sum_{y \in X} \xi (y) I_A (x,y) \ ,
\end{multline*}
where $I_A$ is the indicator function: $I_A(x,y) = 1$ if $y \succeq x$ and $I_A(x,y) = 0$, otherwise.

Therefore
 \begin{multline*}
\normat {\ArensEells} {\xi}  = \sum_{x \in X \setminus{x_0}} \sum_{y \in X} d(x,x^+) \ \Varsig (\Xi(x)) \xi (y) I_A (x,y) = \\ \sum_{y \in X} \xi (y)  \sum_{x \in X \setminus{x_0}} d(x,x^+)  \ \Varsig (\Xi(x)) I_A (x,y) \ ,
\end{multline*}
which is \cref{eq:MAT}.
\qed \end{proof}

\begin{proof}[Second proof of \Cref{th:tree}]

Thanks to the characterization of the extreme points of $\operatorname {Lip}_1^+(d)$ stated in \Cref{prop:exttree},  
we want to maximize the functional
\begin{equation*}
\sum_{y \in X \setminus{x_0}} \xi(y) u_\epsilon (y) =   \sum_{y \in X \setminus{x_0}} \xi(y)\sum_{y \succeq x} d(x,x^+) \epsilon(x),   
\end{equation*}
over all $\epsilon: X \setminus{x_0} \rightarrow \set{-1,1}$. 

On the other hand, by interchanging the order between the two summations,
\begin{equation*}
\sum_{y \in X \setminus{x_0}} \xi(y)\sum_{y \succeq x} d(x,x^+) \epsilon(x) = \sum_{x \in X \setminus{x_0}}d(x,x^+)\epsilon(x) \Xi(x)  \end{equation*}
and so the maximum value is attained when $\epsilon (x) = \Varsig\Xi(x)$ for all $x \in X \setminus{x_0}$. Hence the maximum value will be given by 
\cref{eq:DIS}.
\qed \end{proof}

\begin{remark}
By inspecting the second proof, we find easily the dual elements aligned to the points $\xi \in \ArensEells(X)$. Namely, for every $ \xi \in \AEspace X$ it holds $\scalarof{\xi}{\bar u} = \normat {\ArensEells} {\xi}$, where $\bar u \in \Extreme  \operatorname {Lip}_1^+(d)$ is given by
\begin{equation}
\label{eq:TUTU}
\bar u(y) = \sum_{x \preceq y} d(x,x^+) \Varsig \Xi(x),\quad  \forall y \in X.
\end{equation}
Multiple solutions to the alignment condition $\scalarof{\xi}{u} = \normat {\ArensEells} {\xi}$ will be due to the indeterminacy of the $\Varsig$ function for the vertices $x$ at which $\Xi(x)$ vanishes. 
\end{remark}

The cumulative-sum method to find the norm of \cref{eq:DIS} can be presented in a matrix form. This is illustrated in \Cref{fig:rooted-tree}. It is based on the iteration the adjacency matrix $E$ of the tree, i.e., the matrix whose entries are $a_{x,y}= 1$, if $x$ and $y$ are adjacent with $x \preceq y$,  and $a_{x,y} = 0$ otherwise.

Computing the finite sum
  \begin{equation*}
E^* = (I-E)^{-1} = \sum_{n=0}^\infty E^n,
\end{equation*}
the addenda of \cref{eq:DIS} appear as rows in the matrix $E^*$.

An explanation of this fact is that there is a linear relationship between the distribution $\xi$ and its cumulative distribution $\Xi$, which can be formulated through the adjacency matrix.
After having labeled the rooted tree and denoted by $\xi$ and $\Xi$ the two resultant column vectors, then  the following two equivalent equations hold
\begin{equation*}
\xi = (I-E)\ \Xi \quad \Longleftrightarrow \quad \Xi = (I-E)^{-1}\xi = E^* \xi.    
\end{equation*}
Indeed, the relation on the left-hand side is the matrix form of the obvious identity
\begin{equation}
\label{eq:CHIL}
\xi (x) = \Xi (x) -  \sum_{y \in \chof x} \Xi(y) \ ,     
\end{equation}
for all $x \in X$.
\medskip

Another noteworthy fact is that the $\CUT$ seminorm coincides with Arens-Eells norm, whenever the graph is a tree.

Every edge $e = \set{x,y}$ of tree $(X,\mathcal{E},w)$ splits the vertex set $X$ into two disjoint connected components $S_e$ and $\bar S_e = X \setminus{S_e}$. In turn, the graph metric $d_{T,w}$ can be decomposed in a unique way as
\begin{equation}
\label{eq:deco}
d_{T,w} = \sum_{e \in \mathcal{E}} w_e \delta_{S_e},
\end{equation}
where $w_e = w(x,y)$. See Prop.~11.1.4 of \citet{deza|laurent:1997}.

Observe further that, although $\delta_{S_e} = \delta_{\bar S_e}$, when dealing with rooted trees, it is convenient to  take $S_e$ to be the member of the partition that does not include the root $x_0$.

\begin{theorem}\label{th:AE-vs-CUT}
For every weighted tree with $d_{T,w} = \sum_{e \in \mathcal{E}} w_e \delta_{S_e}$, the \AE-norm is equal to the cut norm of the sequence $C = \set{w_e}$,
\begin{equation}
\label{eq:NCUT}
\normat {\ArensEells} {\xi}  =\normat {C} {\xi}  = \sum_{e \in \mathcal{E}} w_e \aval{\sum_{x \in S_e} \xi(x)} \ .
\end{equation}
\end{theorem}
\begin{proof}
According to the above observation, suppose $x_0 \notin S_e$ for every edge $e$. Clearly edges $e$ are in  one-to-one correspondence with the pair $(x,x^+)$ for $x \in X \setminus{x_0}$.
Likewise,
\begin{equation*}
 \set{y \in X : y \succeq x} = S_{(x,x^+)} \ . \end{equation*}
Therefore, \cref{eq:DIS} equals the expression in \cref{eq:NCUT}. 
\qed \end{proof}

\medskip

The construction of the extreme points for trees made in \Cref{prop:exttree} suggests the following extension.

Associate with every function $\phi$ defined on $X \setminus{x_0}$, the following Kantorovich potential
\begin{equation*}
u_\phi(y) = \sum_{x \preceq y} d(x,x^+)\phi(x)
\end{equation*}
defined on the vertices of the tree.
\begin{proposition}
\label{prop:DUALI}
The mapping $\phi  \mapsto u_{\phi}$, sending $l_\infty (X \setminus{x_0})$ onto $\operatorname {Lip}^+(d)$, is an isometric isomorphism. Its inverse, $\Delta : \operatorname {Lip}^+(d) \mapsto l_\infty (X \setminus{x_0}) $ is given by 
\begin{equation*}
 (\Delta u)(x)= \frac{u(x)-u(x^+)}{d(x,x^+)}   
\end{equation*}
with $(\Delta u)(x_0)=0$.
\end{proposition}
\begin{proof}
Clearly the map is linear. Let us check that it is an isometry. Consider adjacent vertices $y_1,y_2$, with $y_1 \succeq y_2$. Then,
\begin{equation}
\label{eq:ISO}
u_\phi(y_1) =d(y_1,y_2)\phi(y_1) + u_\phi(y_2).
\end{equation}
Hence, $u_\phi(y_1) - u_\phi(y_2) =d(y_1,y_2)\phi(y_1)$,  and so $\normof{u_\phi}_{Lip} \leq \normof{\phi}_\infty$.

On the other hand, if $y_1$ is an element in $X \setminus{x_0}$ for which $\phi(y_1) = \pm \normof{\phi}_\infty$, then the relation \cref{eq:ISO} implies the equality $\normof{u_\phi}_{\Lipschitz} = \normof{\phi}_\infty$. We have so proved that the mapping is an injective isometry.

Denoting by $\Psi$ the direct map $\phi  \mapsto u_{\phi}$, we have
\begin{equation*}
( \Psi \circ \Delta)u (y) = \sum_{y \succeq x \neq x_0} d(x,x^+) \frac{u(x)-u(x^+)}{d(x,x^+)} = u(y).  
\end{equation*}
Consequently, $\Psi$ is onto with inverse given by $\Delta$.
\qed \end{proof}

Let us outline a few consequences that can be derived from the construction of the previous map.

\begin{enumerate}[label = \roman*)]

\item \Cref{prop:DUALI} provides a simple proof that $\operatorname {Lip}^+(d)$ is a dual space. Actually, $\operatorname {Lip}^+(d) \simeq l_1 (X \setminus{x_0})^*$.

\item For every tree with $n$ vertices, $2^{n-1}$ is the number of the extreme points of the unit ball of $\operatorname {Lip}^+(d)$. Actually, it is the image of the unit cube $\normof{x} \leq 1$ of $l_\infty (X \setminus{x_0})$.

\item Interestingly, the inverse map $\Delta$ of $\phi  \rightarrow u_{\phi}$ is closely related to De Leeuw's map (see \citet{weaver:2018-LA-2nd-ed}) which associates with every Lipschitz function $f : X \rightarrow \reals $, the function
\begin{equation*}
 (x,y) \mapsto \frac{f(x)-f(y)}{d(x,y)}   
\end{equation*}
defined for $x \neq y \in X$.   

\end{enumerate}

Another fact of interest is the differentiability of the \AE-norm, which is a direct consequence of the representation (\ref{eq:MAT}).

\begin{proposition}
The norm-function $ f(\xi) = \normat {\ArensEells} {\xi}$ is differentiable at every $\xi \in \AEspace X$ such that $\Xi(x) \neq 0$ for all $x \neq x_0$. Its gradient $\nabla \normat {\ArensEells} {\xi} \in  \Extreme \operatorname{Lip}_1^+(d)$ is given by
\begin{equation*}
\nabla \normat {\ArensEells} {\xi} = \bar u_\xi     
\end{equation*}
where $\bar u_\xi$ is aligned with $\xi$, i.e., $\scalarof{\xi}{\bar u_\xi} = \normat {\ArensEells} {\xi}$. More generally, the directional derivative of $f$ at $\xi$, into the direction $\eta \in \AEspace X$, is 
\begin{multline*}
f'(\xi;\eta)= \\ \sum_{y \in X} \eta(y) \sum_{y \preceq x, x \in V} d(x,x^+) \Varsig \Xi(x) + \sum_{y \in X} \eta(y) \sum_{y \preceq x, x \in V^0} d(x,x^+) \Varsig (\sum_{x \preceq y}\eta(y)).
\end{multline*}
where $V^0 \subseteq X$ is the set of vertices $x$ for which $\Xi(x)$ vanishes and $V = X \setminus V^0$.
\end{proposition}
\begin{proof}
If $\Xi(x) \neq 0$ for all $x \neq x_0$, then the functions $\Varsig \Xi (x)$ are constant in a neighborhood of $\xi$. By \cref{eq:MAT} it follows that the function $\xi \rightarrow \normat {\ArensEells} {\xi}$ is locally linear.  \cref{eq:TUTU} provides the desired gradient.

Tedious algebra leads to the directional derivatives too.
\qed \end{proof}

Though searching optimal solutions is a well-settled  LP problem---see \S~3.3 of \citet{peyre|cuturi:2019}---it is hard to provide a closed-form of the set of solutions as a function of the measure $\xi = \mu - \nu$.

The next result gives a partial positive answer, which is valid under the restriction given in \cref{eq:BABA} below.

\begin{theorem}\label{th:optimalplan}
Consider a rooted tree in $X$ and let $\xi = \mu - \nu$. An optimal coupling $\gamma^* \in \mathcal{P}(\mu,\nu)$ is given by 
\begin{gather*}
\gamma^*(x,x^+) = [\Xi(x)]^+ \ , \quad \gamma^*(x^+,x) = [\Xi(x)]^- \ , \\ \gamma^*(x,x) = \mu(x) -[\Xi(x)]^+ - \sum_{u \in \chof x} [\Xi(u)]^- \ ,
\end{gather*}
and $\gamma^*(x,y) = 0$ otherwise, provided 
\begin{equation}
\label{eq:BABA}
\mu(x) \geq [\Xi(x)]^+ + \sum_{u \in \chof x} [\Xi(u)]^-     
\end{equation}
is true for every $x \in X$. A sufficient condition for \cref{eq:BABA} to hold is that $\mu \gg 0$ and $\normof{\mu - \nu}_{l_1}$ is sufficiently small. 
\end{theorem}

\begin{remark}
  \begin{enumerate}[label = \roman*)]
\item Another class of optimal solutions can be established by changing the role of the two probability functions. Therefore we also have the solution:
\begin{gather*}
\gamma^*(x,x^+) = [\Xi(x)]^- \ , \quad \gamma^*(x^+,x) = [\Xi(x)]^+ \ , \\ \gamma^*(x,x) = \nu(x) -[\Xi(x)]^- - \sum_{u \in \chof x} [\Xi(u)]^+ \\
\text{if} \quad \label{eq:REST}
\nu(x) \geq [\Xi(x)]^- + \sum_{u \in \chof x} [\Xi(u)]^+ \ . 
\end{gather*} 
\item Summing up in \cref{eq:BABA} with respect to the $x$ variable, we get the necessary condition 
\begin{equation*}
    \sum_{x \in X} \aval{\mu(x) - \nu(x)} = \normof{\mu - \nu}_{l_1} \leq 1
\end{equation*}
that obliges the two probability functions to be sufficiently close to each other.
\item It is not hard to check that a sufficient condition to hold \cref{eq:BABA} is that
\begin{equation*}
   \min_x \mu(x) \geq 2\normof{\mu - \nu}_{l_1}.  
 \end{equation*}
\end{enumerate}
\end{remark}

\begin{proof}
  By construction, $\gamma^*(x,y) \geq 0$. Let us show that the plan $\gamma^*$ is feasible. Actually,
  \begin{multline*}
\sum_{y \in X} \gamma^*(x,y) = \gamma^*(x,x) + \gamma^*(x,x^+) + \sum_{u \in \chof x} \gamma^*(x,u) = \\
\gamma^*(x,x) + [\Xi(x)]^+ + \sum_{u \in \chof x} [\Xi(u)]^- = \mu(x), 
\end{multline*}
while
\begin{multline*}
\sum_{x \in X} \gamma^*(x,y) = \gamma^*(y,y) + \gamma^*(y^+,y) + \sum_{u \in \chof y} \gamma^*(u,y) = \\= \mu(y) -[\Xi(y)]^+ - \sum_{u \in \chof y} [\Xi(u)]^- + [\Xi(y)]^- + \sum_{u \in \chof y} [\Xi(u)]^+\\
= \mu(y) - \Xi(y) + \sum_{u \in \chof y} \Xi(u) = \mu(y) - \xi(y) = \nu(y) \ .
\end{multline*}
So we have checked that $\gamma^*$ is a feasible plan.

Regarding its optimality, we have
\begin{align*}
 \sum_{x,y \in X} d(x,y) \gamma^* (x,y) &=  \sum_{x \in X} \sum_{y \in X}d(x,y) \gamma^* (x,y) \\ &=  \sum_{x \in X} [d(x,x^+)\gamma^* (x,x^+) + \sum_{u \in \chof x}  d(x,u) \gamma^*(x,u)]  \\ &= \sum_{x \in X} [d(x,x^+)[\Xi(x)]^+ + \sum_{u \in \chof x}  d(x,u) [\Xi (u)]^-] \\
 &= \sum_{x \in X} d(x,x^+)[\Xi(x)]^+ + \sum_{x \in X} \sum_{u \in \chof x}  d(x,u) [\Xi (u)]^-.
\end{align*}
Under the usual interchanging of summation order, the last addendum becomes
\begin{multline*}
\sum_{x \in X} \sum_{u \in \chof x}  d(x,u) [\Xi (u)]^- = \sum_{x \in X} \sum_{u \in X}  d(x,u) [\Xi (u)]^-I(x,u) = \\
\sum_{u \in X} [\Xi (u)]^- \sum_{x \in X}  d(x,u) I(x,u) = 
\sum_{u \in X} [\Xi (u)]^- d(u^+,u) \ .
\end{multline*}

At last, we get
\begin{equation*}
 \sum_{x,y \in X} d(x,y) \gamma^* (x,y)= \sum_{x \in X} d(x,x^+) \aval{\Xi(x)} =  \normat {\ArensEells} {\xi}   
\end{equation*}
which is the desired result.

\qed \end{proof}

\begin{example}[Barycenter]
\label{ex:barycentre}

If $\mu$ is a probability function defined on the vertices $X$ of a weighted tree, a barycentre is a vertex $\hat x \in X$ such that the K-distance between $\mu$ and the delta function of that vertex is minimal, namely
\begin{equation*}
\normat {\ArensEells} {\delta_{\hat x} - \mu} = \min_{x \in X}   \normat {\ArensEells} {\delta_x - \mu} \ .   
\end{equation*}
Barycenters of probability measures on metric spaces are used in various statistical applications. See \citet{evans|matsen:2012} for the specific example of weighted trees.  

If $\bar x$  denotes the root of the tree, \cref{eq:DIS} yields
\begin{equation*}
\label{eq:bary}
 \normat {\ArensEells} {\delta_{\bar x} - \mu} = \sum_{x \in X \setminus {\bar x}} d(x,x^+)  \sum_{x \preceq y}  \mu(y),        
\end{equation*}
that can be simplified by interchanging the two summations. More specifically, we have
\begin{equation*}
 \normat {\ArensEells} {\delta_{\bar x} - \mu} = \sum_{x \in X \setminus{\bar x}} \sum_{y \in X}  d(x,x^+) \mu (y) I_A(x,y),           
\end{equation*}
where $I_A$ is the indicator function with $I_A(x,y) = 1$ if $x \preceq y$ and $I_A(x,y) = 0$ otherwise.

Therefore,
\begin{multline*}
\normat {\ArensEells} {\delta_{\bar x} - \mu} = \sum_{y \in X} \mu (y) \sum_{x \in X \setminus{\bar x}}  d(x,x^+)  I_A(x,y) = \\ \sum_{y \in X} \mu (y) \sum_{x \preceq y} d(x,x^+)
=\sum_{y \in X} \mu (y) d(y,\bar x) = \expectat \mu {d(\cdot,\bar x)} \ .
\end{multline*}

Consequently, the barycenter $x_B$ will be that vertex that minimizes the $\mu$-mean distance of vertices from itself, i.e.,
\begin{equation*}
\label{eq:bary+}
 x_B = \arg \min_{\bar x \in X} \expectat \mu {d(\cdot ,\bar x)}.    
\end{equation*}
\end{example}

It is worth noticing that an analogous result remains valid if the barycenter is made with respect to C-norms generated by CUT metrics $ d = \sum_S  \lambda_S \delta_S$. Actually,
\begin{equation*}
 \normat{C}{\mu - \delta_ {\bar x }} = \sum_{S \in F_{\bar x}}   \lambda_S (1- \mu(S)) + \sum_{S \notin F_{\bar x}} \lambda_S \mu(S) = \expectat \mu  {d(\cdot ,\bar x)} 
\end{equation*}
where $F_{\bar x}$ is the principal filter generated by the point $\bar x$.

For instance, the barycenter is true for all tree-like metric spaces. For example, we have
\begin{equation*}
\label{eq:bary++}
 x_B = \arg \min_{x \in X} [\phi(x) (1-2\mu(x))]
\end{equation*}
for the metric space  $(X,\phi)$ of \Cref{remark:SSS}.

\subsection{Tree-like spaces}
\label{subsect: tree-like}
Let us first briefly clarify the decomposition formula established in 
\Cref{th:AE-vs-CUT} about trees. We want to clarify which is the underlying structure for the family C associated to a tree.

Given a finite pointed set $(X,x_0)$, where $x_0$ is a distinguished point of $X$, let $\mathcal{S}(X)$ be a family of nonempty subsets of $X$ enjoying the following properties:

\begin{enumerate}[label = \roman*)]
\item  $x_0 \notin S$ for all $S \in \mathcal{S}(X)$ ;
\item  if $S_i, S_j \in \mathcal{S}(X)$ and $S_i \cap S_j \neq \emptyset$, then either $S_i \subseteq S_j$ or $S_j \subseteq S_i$ ;
\item  for all $S \in \mathcal{S}(X)$ it holds $ \# [S \setminus{c(S)} ] = 1$, where $c(S) = \cup [S_i: S_i \subset S]$. \end{enumerate}
Notice that the previous assumptions suggest a tree structure on $X$.

\begin{proposition}
\label{prop:family}
A family $C = \set{\lambda_S } $ provides a distance definable by a weighted tree $T = (X,w)$, with root $x_0$, if, and only if, $\lambda_S > 0$ is equivalent to $S \in \mathcal{S}(X)$, where $\mathcal{S}(X)$ satisfies all the assumptions i)--iii) listed above. 
\end{proposition}

\begin{proof}
According to \cref{eq:deco}, if $T$ is a rooted tree, set $\mathcal{S}(X) = \set{S_{(x,x^+)}, \forall x \neq x_0 }$, where $S_{(x,x^+)} = \set{y \in X : y \succeq x}$. This family satisfies i)--iii). Setting $\lambda_{S_{(x,x^+)}} = d(x,x^+) $, we get the desired implication.

Conversely, let $\mathcal{S}(X)$ satisfy i)--iii). By iii) there is a distinguished point $x = \phi(S) \in S \setminus{c(S)}$, for each $S \in \mathcal{S}(X)$. 
Moreover, given $S \in \mathcal{S}(X)$, consider the collection of the elements $ S_i \in \mathcal{S}(X)$ for which $S \subset S_i$. Condition ii) implies they forms a finite chain having a minimal element $\tilde{S}$, provided the collection is nonempty.

The construction of the tree is carried out as follows. $X$ is set of vertices; the edges of the tree are all the pairs $e = (\phi(S),\phi(\tilde{S}))$, by adding the pairs $(\phi(S),x_0)$, whenever the above collection is empty. Likewise, the weight $w(x,y)$ of the edge (x,y) will be defined as $w(x,y) = \lambda_S$, where $x = \phi(S)$.

While skipping details, we claim it is easy to check that this construction leads to a tree satisfying the conditions of the proposition. \qed
\end{proof}

Let now $Y$ be a tree-like space. It is not restrictive to assume $Y$ to be a subset of a tree $T = (X,w)$. From \Cref{th:AE-vs-CUT} and \Cref{prop:family}, it follows that: 
\begin{equation*}
\normat {\ArensEells (X)} {\xi}  =\normat {C} {\xi}  = \sum_{S \in \mathcal{S}(X)} \lambda_S  \aval{\sum_{x \in S} \xi(x)} \
\end{equation*}
for $\xi \in \ArensEells (X)$ and for some family $\mathcal{S}(X)$.

If in these equations we set $\xi = i_{\#} \eta$, where $i_{\#}: M_0(Y) \rightarrow M_0(X)$ is the canonical push-forward associated with the immersion $i: Y \rightarrow X$, by recalling that $\normat {\ArensEells (X)} {i_{\#} \eta} =  \normat{\ArensEells (Y)} { \eta}$, we get
\begin{equation}
\label{eq:SSS}
\normat {\ArensEells (Y)} {\eta}   = \sum_{S \in \mathcal{S}(X)} \lambda_S  \aval{\sum_{x \in S \cap Y} \eta(x)} \ .
\end{equation}
Hence the $\ArensEells$-norm for a tree-like space $Y$ is a C-norm generated by the trace-family $\mathcal{S^*}(Y) = \set{S \cap Y : S \in \mathcal{S}(X) }$. Notice though that property iii) fails for $\mathcal{S}^*(Y)$. Nevertheless, property ii) remains true, while property i) might be made true. It is enough to pick $x_0$ in $Y$ and select $\mathcal{S}(X)$ to be $x_0$-adapted.

A full treatment of the issues mentioned below would require too much space and would not be consistent with the scope of the present paper. Therefore, we limit ourselves giving the main concepts only. The interested reader is referred the to the quoted papers for a more in-depth analysis.

\citet{bandelt|dress:1992} developed a theory that permits to decompose every finite distance $d$ in a unique way. More specifically, its canonical decomposition is of the kind
\begin{equation*}
    d = d_0 + \sum_{\delta_S \in \Sigma_d} \alpha_d (S) \delta_S
\end{equation*}
where $d_0$ is called the \emph{split-prime residue}  of $d$, the coefficients $ \alpha_d (S)$ are strictly positive numbers, called \textit{isolation index}, while $\Sigma_d$ denotes the set of semi-metrics $\delta_S$ which are $d$-splits. Cf. \S 11.1.2 of \citet{deza|laurent:1997}.

A distance $d$ is said to be \emph{totally decomposable} if $d =  \sum_{\delta_S \in \Sigma_d} \alpha_d (S) \delta_S$ holds. That is, if in the canonical decomposition there is no split-residue, i.e., $d_0 = 0$. Clearl, a totally decomposable metric is $\ell_1$-embeddable but the converse implication is false.

According to this setting, we shall say that a C-norm on $M_0(X)$ is \emph{canonical} if it is induced by a totally decomposable distance $d =  \sum_{\delta_S \in \Sigma_d} \alpha_d (S) \delta_S$, and
\begin{equation*}
 \normat{C}{\xi} = \sum_{\delta_S \in \Sigma_d} \alpha_d (S) \aval{\sum_{z \in S} \xi(z)}  .   
\end{equation*}

To see an example, think that the discrete metric $d$ generated by the complete graph $K_n$ may admit  several decomposition for $n \geq 4$ (see \Cref{ex:1}), but its canonical decomposition is $d = \frac{1}{2} \sum_{x \in X} \delta_{\set{x}}$, as will be seen soon. Hence the metric in $K_n$ is totally decomposable and its canonical norm is given by $\normat{C}{\xi} = \frac{1}{2} \sum_{x \in X} \aval{\xi(x)}$.

Next statement provides the relation between norms for tree-like spaces. Recall that both trees and tree-like spaces have totally decomposable metrics.

\begin{theorem}
For every tree-like space $Y$ we have $\normat {\ArensEells (Y)} {\cdot} = \normat {C} {\cdot}$, where $\normat {C} {\cdot}$ is the canonical C-norm associated with the distance in $Y$.
\end{theorem}

\begin{proof}
By \Cref{prop:family} and the discussion above, we know that the distance $d$ in $Y$ admits the decomposition
\begin{equation}
\label{eq:canon}
d = \sum_{S \in \mathcal{S}(X)} \lambda_S \delta_{ \set{S \cap Y}} \ .  
\end{equation}
Hence it is a \emph{minor} of the totally decomposable metric $d' = \sum_{S \in \mathcal{S}(X)} \lambda_S \delta_{S} $. As totally decomposability is obviously preserved by taking minors, we infer that the C-norm associated with the decomposition \eqref{eq:canon} is canonical. Consequently, the equality in \cref{eq:SSS} provides the desired result. \qed
\end{proof}

\section{Beyond trees}\label{sec:extensions}

In this last section, we study a few extensions along two distinct lines of research both suggested by the previous results on trees.

A first extension is based on computing the K-distance through the spanning trees of a given arbitrary graph. \citet{mendivil:2017} has suggested a different approach, based on starting from a single spanning tree and then reach the full graph via a quotient map.

The distance induced by trees is a rigid $\ell_1$-embeddable metric. A second development is thus related to the study the extend to which the results for trees can be generalized to other types of $\ell_1$-embeddable metrics.

\subsection{Spanning trees}
\label{sec:spanning-tree}
If $G = (X,\mathcal E)$ is a connected graph, a spanning tree of $G$ is a tree $T = (X,\mathcal T)$ with $\mathcal T \subset \mathcal{E}$. In other words, $T$ is a sub-graph of $G$ with the same vertex set as $G$ and with the minimum number of edges that allows connection. See, for example, \S~1.2 of \citet{bollobas:1998}.

The inclusion relation $\mathcal T \subset \mathcal{E}$ implies the inequality $d \leq d_T$ between the two distances $d$ and $d_T$ induced by $G$ and $T$, respectively. Hence, $\normat{G}{\xi} \leq \normat{T}{\xi}$ holds for the two graphs. 

Denoting by $\operatorname{ST}(G)$  the totality  of the spanning trees of $G$, it follows that
\begin{equation*}
\normat{G}{\xi} \leq \min_{ T\in \operatorname{ST}(G)} \normat{T}{\xi} \ .    
\end{equation*}

In order to show that the above inequality is in fact an equality, we need the following lemma about growing a forest to a tree. The result is provided by the Kruskal algorithm,  \citet{kruskal:1956}. See also \cite[p. 10]{bollobas:1998}.

\begin{lemma}
\label{lemma:Fore}
Let G be a connected graph and $F $ be a forest contained in $G$. There exists a spanning tree of $G$ which extends $F$.
\end{lemma}

\begin{theorem}
\label{prop:ENV}
The Arens-Eells norm of any connected graph $G$ is the envelope of the norms of its spanning trees. That is,
\begin{equation*}
  \normat {G} {\xi} = \min_{T \in \operatorname{ST}(G)} \normat{T}{\xi}.
\end{equation*}
\end{theorem}

\begin{proof}
\Cref{prop:Vague} implies that the Arens-Eells norm of $\xi$ is
\begin{equation*}
\normat {\ArensEells} {\xi} = \sum_{(x,y) \in \mathcal{F}} \aval{a(x,y)} d(x,y)  
\end{equation*}
where $F = (X,\mathcal{F})$ is an a-cyclic subgraph of $G = (X,\mathcal{E})$. In other words, $(X,\mathcal{F})$ is a forest. By \Cref{lemma:Fore}, there is a spanning tree $T = (X, {\mathcal T})$ extending such a forest. If we enlarge the domain of the functions $a(x,y)$ to $(x,y) \in {\mathcal T}$, by assigning the value $a(x,y) = 0$, outside $\mathcal{F}$, we can re-write the equation above as
\begin{equation*}
\normat {\ArensEells} {\xi} = \sum_{(x,y) \in {\mathcal T}} \aval{a(x,y)} d(x,y) = \sum_{(x,y) \in {\mathcal T}} \aval{a(x,y)} w(x,y),
\end{equation*}
where the last equality follows from \cref{item:vague-2} of \Cref{prop:Vague}, since the pair of vertices $x$ and $y$ are close, as long as $a(x,y) \neq 0$.

To conclude,
\begin{equation*}
\normat {\ArensEells} {\xi} = \sum_{(x,y) \in {\mathcal T}} \aval{a(x,y)} w(x,y) \geq \normat {T} {\xi} \geq \min_{T \in \operatorname{ST}(G)} \normat{T}{\xi}
\end{equation*}
that proves our assertion.
\qed \end{proof}

\subsection{A worked out example: cycle graphs}\label{sec:cycle}
It should be of some interest to solve by hand a few examples of what was stated in \Cref{prop:ENV}. The cyclic case has already been analyzed by \citet{cabrelli|molter:1995} as well as \citet{mendivil:2017} but by quite different techniques.

A labelled weighted cyclic graph (or circuit) of order $n$, denoted by $C_n$, is the graph
\begin{equation*}
1 \rightarrow 2 \rightarrow \dots \rightarrow n \rightarrow 1
\end{equation*}
consisting of a unique cyclic path. Set $d_i = d(i,i+1) $ the distance between the two adjacent vertices $i$ and $i+1$.

Clearly the cycle graph $C_n$ admits $n$ spanning trees $\set{T_i}_{i=1}^n$ obtained by ruling out each single edge of $C_n$.

Next proposition provides explicitly the Arens-Eells norm $\normof{\cdot}_{C_n}$ for the cycle $C_n$ as well as a constructive proof of the envelope property.

\begin{proposition}
\label{prop:CYL}
Let $C_n$ be the cycle graph of order $n$ with vertices $1,2,\dots,n$ and edges $\set{i,i+1}$, $n+1=1$.  Define the real function
\begin{equation*}
 \Phi(t) = \sum_{i=1}^n \aval{t - \xi_1 -\xi_2 - \dots - \xi_i} d_i \ , \quad t \in \reals \ .  
\end{equation*}
Then, for each $\xi \in \AEspace X$,
\begin{equation*}
\normof{\xi}_{C_n} = \min_{t \in \reals} \Phi(t) = \min_{i=1,2,..,n} \Phi (\xi_1 +\xi_2 + \dots + \xi_i) = \min_{i=1,2,..,n} \normof{\xi}_{T_i}  
\end{equation*}
where $T_i \in \operatorname{ST}(C_n)$. Specifically, $\Phi (\xi_1 +\xi_2 + \dots + \xi_i)$ is the norm for the tree obtained by removing the edge $\set{i-1,i}$.
\end{proposition}

As observed by \citet{mendivil:2017}, the value $t$ that minimizes $\Phi$ is the weighted median value of the distribution $\xi$.

\begin{proof}
  Thanks to \Cref{item:vague-2} of  \Cref{prop:Vague}, the restriction of the elements $a(x,y)$ leads to the representations
  \begin{equation*}
  \xi = a_{1} (\delta_1 - \delta_2) + a_{2} (\delta_2 - \delta_3) + \dots + a_{n} (\delta_n - \delta_1)  
\end{equation*}
with $a_{i} \in \reals$.
By inverting the previous relation and introducing the parameter $t= -a_{n}$, we get easily that
\begin{equation*}
    a_{i} = t - \xi_1 - \xi_2 - \dots - \xi_i
\end{equation*}
for $i= 1,2, \dots, n$. This implies that every vector $\xi$ admits $\infty^1$-many representations, and Arens-Eells formula \cref{eq:AEnorm} for the norm becomes $\inf_{t \in \reals} \Phi(t)$.

Of course, this piece-wise linear and convex function $\Phi$ reaches the minimum value at one of the $n$ points $t = \xi_1+\xi_2- \dots + \xi_i$, ($i= 1,2,\dots,n$), and so also the second formula is checked. 

It remains to show that the values $\Phi (\xi_1 +\xi_2 + \dots + \xi_i)$ are nothing but the Arens-Eells norms of the spanning trees of $C_n$.

Fix an index $j$ and evaluate the function $\Phi$ at the point $\xi_1 + \xi_2 +\cdots +\xi_j$, then
\begin{equation*}
\Phi (\xi_1 +\xi_2 + \dots + \xi_i) = \sum_{i=2}^{i=j} \aval{\xi_i + \dots + \xi_j}d_{i-1} + \sum_{k=1}^{k=n-j} \aval{\xi_{j+1} + \dots + \xi_{j+k}}d_{j+k}   
\end{equation*}
If now we get rid of variable $\xi_j$, by means of the relation $\xi_j = - \sum_{i \neq j}\xi_i $, it is not difficult to check that $\Phi (\xi_1 +\xi_2 + \dots - \xi_i)$ turns out to be the norm of the linear tree
\begin{equation*}
  j \rightarrow j+1 \rightarrow \dots \rightarrow n \rightarrow 1 \rightarrow \dots \rightarrow j-1
\end{equation*}
by taking $j-1$ as root.
\qed \end{proof}

The method employed in \Cref{prop:CYL}, might be duplicated for other graphs for which $\# X = \# \mathcal E = n$, like in the cycle graphs. However, the case  $\# \mathcal E > n$ is more interesting and clearly the function $\Phi$, in this case, will be no longer a scalar one.

\begin{example}\label{example:diagonal}
By way of example, consider the two-cycles graph
$\vcenter{\hbox{\begin{tikzpicture}
  [scale=.15,auto=left]
  \node (n1) at (-5,0) {1};
  \node (n2) at (0,-5)  {2};
  \node (n3) at (5,0)  {3};
  \node (n4) at (0,5) {4};
  \foreach \from/\to in {n1/n2,n2/n3,n3/n4,n4/n1,n2/n4}
  \draw (\from) -- (\to);
\end{tikzpicture}}}$
with $\# X =4$, $\# \mathcal E = 5$, and unit weight. The function to minimize turns out to be
\begin{equation*}
\Phi (t,u) = \aval{u} + \aval{t} + \aval{\xi_1 - u} + \aval{\xi_1 + \xi_2 -t +u} + \aval{u-t -\xi_4} \ , \quad (t,u) \in \reals^2 \ .  
\end{equation*}

After tedious algebra, the norm of the vector $\xi =(\xi_1,\xi_2,\xi_3,\xi_4)$ turns out to be the minimum of the following 8 functionals:
$\aval{\xi_1} + \aval{\xi_2} + \aval{\xi_1 + \xi_4}$, \ $\aval{\xi_1} + \aval{\xi_4} + \aval{\xi_1 + \xi_2}$, \ $\aval{\xi_1} + \aval{\xi_3} + \aval{\xi_1 + \xi_4}$, \ $\aval{\xi_1} + \aval{\xi_3} + \aval{\xi_1 + \xi_2}$, \ $\aval{\xi_2} + \aval{\xi_3} + \aval{\xi_1 + \xi_2}$, \ $\aval{\xi_3} + \aval{\xi_4} + \aval{\xi_1 + \xi_4}$, \ $\aval{\xi_1} + \aval{\xi_3} + \aval{\xi_4}$, \ $\aval{\xi_1} + \aval{\xi_2} + \aval{\xi_3}$, \
which are just the norms of the $8$ spanning trees of the graph.
\end{example}

\subsection{$\ell_1$-embeddable metrics}

With the notations of \Cref{sub:CUTS}, fix a distance $d \in \CUT(X)$, i.e., let $d$ be an $\ell_1$-embeddable distance having realization $d = \sum_{S \subseteq X} \lambda_S \delta_S$, and fix a distinguished point $x_0$ of $X$. For $C = \set{\lambda_S}$, the norm $\normat C \cdot$ was defined in \Cref{def:cut-norm}.
We suppose here that the decomposition be $x_0$-adapted.

Define the family of functions $\epsilon : 2^X \rightarrow \set{-1,+1}$. To each such $\epsilon$, we can attach the function $u_\epsilon : X \rightarrow \reals$ defined by
\begin{equation}\label{eq:u-epsilon}
    u_\epsilon (x) = \sum_{S \in \mathcal{F}_x} \lambda_S \ \epsilon (S) \quad \forall x \in X \ ,
\end{equation}
where $\mathcal{F}_x$ is the principal filter generated by the point $x \in X$. Observe that $u_\epsilon (x_0) = 0$, since the representation is $x_0$-adapted.

Clearly, $u_\epsilon \in   \operatorname {Lip}_1^+(d)$ for every function $\epsilon$, but not necessarily all of them are extremal functions of the unit ball. The theorem below reveals a relation between the extremal functions of the unit ball and the class of $u_\epsilon$ functions. The proof is, in fact, obvious for trees and the general case parallels the second proof of \Cref{th:tree}.
\begin{theorem}
\label{th:EMB}
It holds $\normof{\cdot}_{C} = \normof{\cdot}_{\ArensEells}$, provided every $u \in \Extreme \operatorname {Lip}_1^+(d)$ is of the form given in \cref{eq:u-epsilon}.
\end{theorem}

\begin{proof}
For every $\xi \in M_0(X)$ and $u_\epsilon$, we have
\begin{multline*}
\scalarof{\xi}{u_\epsilon} = \sum_{x \in X} \xi (x) u_\epsilon (x) = \sum_{x \in X} \xi (x) \sum_{x \in S} \lambda_S \epsilon (S) = \\ \sum_{x \in X} \xi (x) \sum_{S \subset X} \lambda_S \epsilon (S) I(x,S) = 
\sum_{S \subset X} \lambda_S \epsilon (S) \sum_{x \in S} \xi (x) \ ,
\end{multline*}
where $I(x,S) = 1$ if $x \in X$ and $I(x,S) = 0$ elsewhere.

To conclude, under our assumptions the norm $\normat{\ArensEells}{\xi}$ will be reached by maximizing $\scalarof{\xi}{u_\epsilon}$ over the class of functions $\epsilon$. Actually, $\epsilon(S) = \Varsig (\sum_{x \in S} \xi (x) )$ and thus $\normof{\xi}_{C} = \normof{\xi}_{\ArensEells}$.
\qed \end{proof}

\begin{example}
In view of the example of \Cref{sub:CUTS} the discrete metric in $X=\set{1,2,3,4}$ has the 1-adapted realization $d = \frac{1}{2} (\delta_{\set{2,3}} + \delta_{\set{2,4}} + \delta_{\set{3,4}} )$.

The sufficient condition of \Cref{th:EMB} fails for the associated $C_2$-norm. Actually, if for the functions $u_\epsilon$ we impose the two conditions $u_\epsilon (1) = 0$ and $u_\epsilon (2) = 1$, we get that necessarily $u_\epsilon (3) = u_\epsilon (4)$ and so there are extremal points not covered by the $u_\epsilon$'s. 
\end{example}

\subsection{Quotient maps}

Through the identification of vertices, \citet{mendivil:2017} compute K-distances of new graphs. For instance, the identification of the end-points of a linear graph leads to  the K-distance for a cycle graph. A few examples are described in that paper, in particular, it is shown that any graph can be seen as the quotient space of a tree.

We extend here to metric spaces such an approach as well as some further clarifications are discussed for graphs. Though our result holds true in more general metric spaces, we assume the metric spaces to be finite as it is in the rest of this paper.

Given a mapping $q: X \rightarrow Y$, the associated push-forward map, $q_{\#}(\xi) = \xi\circ q^{-1}$, acts, by restriction on the 0-mass probability functions, as $q_{\#}: M_0(X) \rightarrow M_0(Y)$. Notice that the push-forward map $q_{\#}$ is surjective if $q$ itself is surjective. In fact, surjectivity of $q$ implies that every delta function on $Y$ is the image of a delta function on $X$, and every probability function is a convex combination of delta functions.

Given metric spaces $(X,d_X)$ and $(Y,d_Y)$, we say that the mapping $q \colon X \to Y$ is \emph{exactly non-expansive} if $d_Y(q(x),q(y)) \leq d_X(x,y)$, $x,y \in X$, and moreover for every $u,v \in Y$, there exist points $x,y \in X$ for which  $d_X(x,y) = d_Y(u,v)$, with $q(x)=u$ and $q(y)=v$. That is, the inequality is actually an equality.

\begin{theorem}
\label{th:quot}
If the mapping $q: X \rightarrow Y$ is surjective and exactly non-expansive then
\begin{equation*}
 \normat{\AEspace{Y}}{\eta} = \inf_{q_{\#}(\xi) = \eta} \normat{\AEspace{X}}{\xi} \ , \quad \eta \in \ArensEells(Y) \ .     \end{equation*}

In other words, the quotient space $\ArensEells(X) / \ker q_{\#}$ is isometric to $\ArensEells(Y)$.  
\end{theorem}

\begin{proof}
Since $q_{\#}$ is surjective, for each $\eta \in \ArensEells(Y)$ there is a $\xi$ such that $q_{\#}(\xi) = \eta$. If  $\xi = \sum a(x,y) (\delta_x - \delta_y)$ is an optimal representation, meaning $\normat{\AEspace{Y}}{\xi} = \sum_{x,y} \aval{a(x,y)}$, then 
 \begin{equation*}
  \eta = \sum a(x,y) (\delta_{q(x)} - \delta_{q(y)}) \ ,   
 \end{equation*}
 which in turn implies
 \begin{equation*}
  \normat{\AEspace{Y}}{\eta} \leq \sum \aval{a(x,y)}d_Y(q(x),q(y)) \leq \sum \aval{a(x,y)}d_X(x,y) =   \normat{\AEspace{X}}{\xi} \ .   
 \end{equation*}

 Since this is true for every $\xi \in (q_{\#})^{-1} (\eta)$, it follows the inequality 
\begin{equation*}
 \normat{\AEspace{Y}}{\eta} \leq \inf_{q_{\#}(\xi) = \eta} \normat{\AEspace{X}}{\xi}      
\end{equation*}
To check that it is an equality, consider that the RHS
\begin{equation*}
 \Phi (\eta) = \inf_{q_{\#}(\xi) = \eta} \normat{\AEspace{X}}{\xi} \end{equation*}
which is clearly a semi-norm of $M_0(Y)$. Now, for each $u,v \in Y$ there are points $x,y \in X$ which realize the exactness assumption, so that
\begin{equation*}
\Phi(\delta_u - \delta_v) \leq  \normat{\AEspace{X}}{\delta_{x} - \delta_{y}} =  d(x,y) = d(u,v) \ .   
\end{equation*}
By \Cref{prop:LARG}, the reverse inequality $\Phi (\eta) \leq  \normat{\AEspace{Y}}{\eta}$ holds and the first claim is proved.

The last statement follows from the definition of quotient norm. \qed
\end{proof}

The condition of being exactly non-expansive seems involved, but it is quite simply verified when the distance is defined by graphs. For example, in the linear graph $x_1  \cdots  x_{n}$, with $q(x_1) = q(x_{n}) = x_1$ and $q(x_j) = x_j$ in all other cases, the map $q$ from the linear graph to the $C_{n-1}$ is surjective and exactly non-expansive as the general argument below shows.

\begin{proposition} \label{prop:identification} Let $q: G \rightarrow H$ be a surjective map  between two weighted graphs and assume the following two conditions both hold:
  \begin{enumerate}[label = \roman*)]
  \item \label{item:identification1} If $x,y$ are close vertices in $G$, then either $q(x) = q(y)$, or $q(x)$ and  $q(y)$ are adjacent vertices  in $H$ and $d_H(q(x),q(y))=d_G(x,y)$;
\item If $u$ and $v$ are close vertices in $H$, then there are vertices $x,y \in G$ for which $q(x)=u$, $q(y) = v$ and $d_G(x,y) \leq d_H(u,v)$.
\end{enumerate}
Then $q$ is exactly non-expansive and hence the conclusion of \Cref{th:quot} holds.
\end{proposition}

\begin{proof} 
 As remarked in \Cref{sec:distance-graph}, a path of minimal distance can be realized by close vertices. Let $x,y \in X$ and $x=x_1,x_2,...,x_n=y$ be such a path. Under \Cref{item:identification1}, the path $q(x_1),q(x_2),...,q(x_n)$ links $q(x)$ to $q(y)$. Hence 
\begin{equation*}
 d_H(q(x),q(y)) \leq \sum_{i} d_H(q(x_i),q(x_{i+1}) \leq  \sum_{i} d_G(x_i,x_{i+1}) = d_G(x,y). \end{equation*}
Therefore $q$ is non-expansive.

Take now two vertices $u,v \in H$ and an existing path $u=u_1,u_2,...,u_m=v$ of minimal distance and composed by close vertices. Under assumption (ii) we get a path $x= x_1,x_2,...,x_m =y$ that satisfies the conditions: $q(x_i)=u_i$ and $d_G(x_i,x_{i+1}) \leq d_H(u_i,u_{i+1})$.
Therefore,
\begin{equation*}
 d_H(u,v) = \sum_{i=1}^{m-1} d_H(u_i,u_{i+1}) \geq \sum_{i=1}^{m-1} d_G(x_i,x_{i+1}) \geq d_G(x,y).
\end{equation*}
Since we already know that $d_H(u,v) \leq d_G(x,y)$, the exactness property is proven. \qed \end{proof}

\section{Conclusion}\label{sec:conclusion}

In this piece of research, we have learned that the dual of the Kantorovich LP problem has a closed form solution in case the ground distance is defined by a tree, and that, in turn, this an instance of the CUT-norm theory. In case of a general graph, the K-distance is the envelope the distances on each of the spanning trees. This, clearly, prompts for further research about the possible computational application. This is not done here but we provide simple tutorial examples. Our approach is closely based on the systematic use of Arens-Eells theory. We believe this aspect is our main contribution to the topic together with some remarks about possible generalisation to other types of metric spaces.  

\section*{Acknowledgments}
\label{sec:acknowledgments}
Both authors thank L. Malag\`o (RIST Cluj-Napoca) for suggesting relevant references. G. Pistone has learned a lot about the state-of-the-art attending a short course on Computational Optimal Transport by G. Peyr\'e. G. Pistone acknowledges the support of de Castro Statistics and of Collegio Carlo Alberto. He is a member of INdAM-GNAMPA.

\bibliographystyle{spbasic}

\end{document}